\documentclass[12pt,reqno]{amsart}
\usepackage[utf8]{inputenc}
\usepackage{graphicx}
\usepackage{amsmath}
\usepackage{amssymb,amsthm}
\usepackage[T1]{fontenc}
\usepackage{pgfplots}
\usepackage{enumerate}

\theoremstyle{plain}
\newtheorem{theorem}{Theorem}
\newtheorem{lemma}[theorem]{Lemma}
\newtheorem{example}[theorem]{Example}
\newtheorem{cor}[theorem]{Corollary}

\theoremstyle{definition}
\newtheorem{deff}[theorem]{Definition}
\newtheorem{rem}[theorem] {Remark}
\newtheorem{algo}[theorem]{Algorithm}
\newtheorem{openq}[theorem]{Question}

\allowdisplaybreaks

\renewcommand{\P}{\mathbb{P}}
\newcommand{\N}{\mathbb{N}}

\newcommand{\R}{\mathbb{R}}

\newcommand{\Z}{\mathbb{Z}}

\renewcommand{\phi}{\varphi}
\renewcommand{\epsilon}{\varepsilon}

\newcommand{\ord}{\mathrm{ord}}

\newcommand{\rad}{\mathrm{rad}}
\newcommand{\Mid}{\mathrel\|}

\begin{document}

\title[Fractions with Euler's Totient Function]%
{Density Properties of Fractions with Euler's Totient Function}
\author{Karin Halupczok and Marvin Ohst}

\address{K.~Halupczok, M.~Ohst: Mathematisches Institut der
Heinrich-Heine-Universit{\"a}t D{\"u}sseldorf, Universit{\"a}tsstr. 1,
40225 D{\"u}sseldorf, Germany}
\email{karin.halupczok@uni-duesseldorf.de, marvin.ohst@uni-duesseldorf.de}

\keywords{Euler's totient function, Schinzel density, Arnold's question}
\subjclass[2020]{11A07, 11N37, 11N56, 11N69}

\begin{abstract}
  We prove that for all constants $a\in\N$, $b\in\Z$, $c,d\in\R$, $c\neq 0$,
  the fractions $\phi(an+b)/(cn+d)$ lie
  dense in the interval $]0,D]$ (respectively $[D,0[$ if $c<0$), 
  where $D=a\phi(\gcd(a,b))/(c\gcd(a,b))$. 
  This interval is the largest possible, since
  it may happen that isolated fractions lie outside of the interval:
  we prove a complete determination of the case where this
  happens, which yields an algorithm that calculates the
  amount of $n$ such that $\rad(an+b)|g$ for coprime $a,b$ and any $g$.
  Furthermore, this leads to an interesting open question which is
  a generalization of a famous problem raised by V.~Arnold.
  For the fractions $\phi(an+b)/\phi(cn+d)$ with constants
  $a,c\in\N,b,d\in\Z$, we prove that 
  they lie dense in $]0,\infty[$ exactly if $ad\neq bc$.
\end{abstract}
\maketitle

\section{Introduction}
\label{sec1}

Let $\phi$ denote Euler's totient function
for which
\[\phi(n)=n\prod_{p\mid n}\Big(1-\frac{1}{p}\Big).\]
In 1954, Schinzel and Sierpi\'nski \cite{schsp}
proved that for positive integers $n$, the fractions $\phi(n)/n$
lie dense in the interval $]0,1]$, and in \cite{schinzel}, Schinzel
proved that
the fractions $\phi(n+1)/\phi(n)$ lie dense in $]0,\infty[$.
Many research papers on related topics followed.

Recently in 2019,  Garcia, Luca,
Shi and Udell proved in \cite{GLSU}
Schinzel-type-theorems on the density of the fractions
$\phi(p + 1)/\phi(p)$ and \linebreak $\phi(p)/\phi(p - 1)$
when $p$ runs over the set of primes. 

Here we consider fractions of the form $\phi(an+b)/(cn+d)$
and $\phi(an+b)/\phi(cn+d)$ with
appropriate constants $a,b,c,d$,
but let $n$ run through
all positive integers that are at least as big as an appropriate starting value.

It can easily be observed that it suffices to restrict to the case $a,c>0$.
The maximal length $D$ of the target interval depends then explicitly
on the parameters $a,b,c$, and on the greatest common divisor
$(a,b)$ of $a$ and $b$. 

\medskip
The following theorem is our first main result.

\begin{theorem}\label{Thm3.1}
Let $a\in\N$, $b\in\Z$, $c\in\R_{>0}$, $d\in\R$ be constants. Put
\[D:=\frac{a\phi((a,b))}{c(a,b)}\]
and
\[M:=\left\{ n\in\N;\ n>\max\Big\{-\frac{b}{a},-\frac{d}{c}\Big\} \text{ and } 
 \frac{\phi(an+b)}{cn+d}\leq D\right\}.\]
Then the set
\[L:=\left\{
\frac{\phi(an+b)}{cn+d};\ 
n\in M
\right\}\]
is dense in the real interval $]0,D]$.
Furthermore, for arbitrary $t\in\N$, we can restrict $L$ by any of the following
conditions on $n$ such that $L$ remains dense in $]0,D]$.
\begin{enumerate}
\item $\frac{an+b}{(a,b)}$ is squarefree,
\item $\frac{an+b}{(a,b)}$ and $t$ are coprime,
\item there is at most one prime factor $p$ of $\frac{an+b}{(a,b)}$
  such that $p\not\equiv1\mod a$, and if it exists, it has
  multiplicity 1 and $a>1$.
\end{enumerate}
\end{theorem}

We give the proof in Section~\ref{sec3}.

Furthermore, we prove a generalization to $n$ in an arbitrary arithmetic
progression $r$ mod $m$,
Corollary~\ref{Cor3.3}, which follows almost directly from
Theorem~\ref{Thm3.1}. We also deduce that the $\limsup$
of the fractions is $D$ and
the $\liminf$ is $0$, stated as Corollary~\ref{Thm3.4}.

\medskip
As our second main result, we prove in Section~\ref{sec4} the following.

\begin{theorem}\label{Thm4.2}
Let $a,c\in\N$ and $b,d\in\Z$ be constants and define
\[L:=\left\{
\frac{\phi(an+b)}{\phi(cn+d)};\ 
n\in\N,\ 
n>\max\Big\{-\frac{b}{a}, -\frac{d}{c}\Big\}
\right\}.\]
If $ad\neq bc$, then $L$ is dense in $]0,\infty[$.
If $ad=bc$, then $\# L=2^{\omega(\tau_c(a)\tau_a(c))}$,
where the exponent is the number of primes dividing only 
one of the numbers
$a,c$, specified by Definition~\ref{def1.1}
and $\omega(n):=\#\{p\mid n\}$.
\end{theorem}

The proof is split into Theorem~\ref{Thm4.2 dense case} and \ref{Thm4.2 not dense case}.

In Section~\ref{sec5.1}, we consider the case when $M=\N_{>-b/a}\cap \N_{>-d/c}$
in Theorem~\ref{Thm3.1}, that is, when there are no fractions outside
the density interval $]0,D]$.
We characterize this condition using
$\rho(a,b)=\#\{n>-b/a;\ \rad((an+b)/(a,b))\mid (a,b)\}$,
and give a full characterization in Theorem~\ref{Thm3.6}.

In Section~\ref{sec5.2}, we determine $\rho(a,b)$ for some pairs.

By Theorem~\ref{Thm3.8}, relevant pairs $a,b$ with $\rho(a,b)=\infty$
are given by the condition
\begin{equation}\label{eq:cond3.8}
  \exists \alpha_{1},\dots,\alpha_{s}\in\N_{0}, \alpha_{1}+\dots+\alpha_{s}>0:
  q_{1}^{\alpha_{1}}\cdots q_{s}^{\alpha_{s}} \equiv \frac{b}{(a,b)}
  \mod\frac{a}{(a,b)},
\end{equation}
provided that $q_{1},\dots, q_{s}$ are the prime divisors
of $(a,b)$ that are coprime to $a/(a,b)$.

These results lead us to algorithms that make these pairs visible by a diagram
for $1\leq a\leq 2001$ and $-1000\leq b\leq 1000$.
Some of the visible lines in this diagram are explained by Theorem~\ref{Thm3.9}.

In Section~\ref{sec5.3}, we prove an asymptotic formula
for the number of $n\leq x$ in the set defining $\rho(a,b)$
in the case when $\rho(a,b)=\infty$. The result is stated
in Theorem~\ref{Thm3.11}.

Finally, Section~\ref{sec5.4} is dedicated to counting the
pairs $a,b$ such that condition \eqref{eq:cond3.8} holds. This leads
to an interesting generalization of Arnold's question on the average
value of the multiplicative order of numbers: We raise the question
if the number of these pairs is $=o(x^{2})$ or not, see the precise
formulation in Question~\ref{openq}.

\bigskip
\textbf{Notation.} $\Z$ denotes the set of all integers,
$\N$ the set of all
positive integers, and let $\N_0:=\N\cup\{0\}$. The symbol $\P$ denotes the set
of all primes (in $\N$), and $\R$ the set of real numbers.
By the letters $p$ and $q$, we always denote primes,
and for $k\in\N$, let $p_k$ denote the $k$-th prime, whereas
$q_k$ stands for a prime factor of some integer.

For $a,b\in\Z$, %
let $(a,b):=\gcd(a,b)$ with the convention $\gcd(0,0):=0$.
For coprime $a,b\in\Z$, $b>0$,
let $\ord_b(a)$ denote the multiplicative order
of $a$ modulo $b$, i.e.\ $\ord_b(a)$ is the smallest positive integer $n$ such
that $a^n\equiv1\mod b$. 
For a prime $p$ and $k\in\N$, $n\in\Z$, we write
$p^k\Mid n$ if $p^k\mid n$ and $p^{k+1}\nmid n$, i.e., $p^k$ is the exact
prime power of base $p$ that divides $n$.

More generally, we write $m\Mid n$ for $m\in\N$,
if $\frac{n}{m}$ and $m$ are coprime.
A substack $p\mid n$ under a sum or product sign means that the sum or
product extends over all prime divisors of $n$, and
a substack $p^k\Mid n$ under a sum or product sign means that the sum or
product extends over all prime powers dividing $n$ exactly.
Additional restrictions may also appear.

For a nonzero integer $n$, we denote the radical or squarefree kernel of $n$ as
\[\rad(n):=\prod_{p\mid n}p.\]

For some $s\in\R$ and two functions $f,g:[s,\infty[ \to\R$, we write
$f(x)\sim g(x)$
    if $\lim_{x\to\infty}f(x)/g(x)=1$, and
    $f(x)=o(g(x))$, if $\lim_{x\to\infty}f(x)/g(x)= 0$.
    We write $f(x)=O(g(x))$ if
    there is a constant $C>0$ such that $|f(x)|\leq Cg(x)$ for all
    $x\in[s,\infty[$. 
    Finally, $\log(x)$ denotes the natural logarithm.


\section{Auxiliary Lemmas}
\label{sec2}

\begin{deff}\label{def1.1}
For $a\in\Z$, $a\neq0$ and $n\in\Z$ let
\[\gamma_n(a):=\prod_{\substack{p^k\Mid a\\p\mid n}} p^k
\text{ and }
\tau_n(a):=\prod_{\substack{p^k\Mid a\\p\nmid n}} p^k,\]
thus $\gamma_n(a)$ is the prime factor part of $a$ that shares all
prime divisors with $n$, and $\tau_n(a)$ is the prime factor part of
$a$ whose prime divisors are coprime to $n$.
\end{deff}
We summarize some basic properties of these arithmetic functions
that can easily
be deduced from their definition and by unique prime factorization.
\begin{lemma}\label{lemma1}For $a\in\Z$, $a\neq0$ and $n\in\Z$ we have
\begin{enumerate}[(i)]
\item $\gamma_n(a)\cdot\tau_n(a)=a$,
\item $\rad(\gamma_n(a))\mid n$,
\item $(\tau_n(a),n)=1$,
\item $\gamma_n(a)=\gamma_{\rad(n)}(a)$ and $\tau_n(a)=\tau_{\rad(n)}(a)$
  (for $n\neq 0$),
\item $(\gamma_n(a),\tau_n(a))=1$.
\end{enumerate}
\end{lemma}

The multiplicative property of $\phi$ can be restated and
generalized in terms of $\gamma_n,\tau_n$ as follows.

\begin{lemma}\label{Thm2.3}
  For all $n,m\in\N$ we have
  \begin{enumerate}[(i)]
  \item
    $\phi(nm)=\phi(n)\phi(m)$
    if and only if $n,m$ are coprime, %
  \item
    $\phi(nm)=n\phi(m)$
    if and only if $\rad(n)\mid m$, %
  \item
    $\phi(nm)=\gamma_n(m)\cdot\phi(\tau_n(m))\cdot\phi(n)$.
  \end{enumerate}
  \end{lemma}

\begin{proof} The first two parts follow immediately from the
  definition of $\phi$ and by unique prime factorization.
  From Lemma~\ref{lemma1}, %
  we know that $m=\gamma_n(m)\cdot\tau_n(m)$,
  $\rad(\gamma_n(m))\mid n$, and $(\tau_n(m),n\gamma_n(m))=1$.
  Thus, the third part follows immediately from the first parts.
\end{proof}

\begin{rem}In this sense, the condition $\rad(n)\mid m$ can be seen as the
 ``opposite'' 
  of coprimality $(n,m)=1$. Note however that the roles of $n$ and $m$
in $\rad(n)\mid m$
cannot always be swapped like in $(n,m)=1$.
\end{rem}

The following version of Mertens' theorem with primes in
arithmetic progression
is used in the proof of Theorem~\ref{Thm3.1}.

\begin{lemma}\label{Thm2.4}
For coprime $a,b\in\N$ and $x\to\infty$, we have
\[
\prod_{\substack{p\leq x\\p\equiv b\mod a}}\left(1-\frac{1}{p}\right)
\sim\frac{C(a,b)}{\log(x)^{1/\phi(a)}}
\]
for a real constant $C(a,b)>0$.
\end{lemma}

A proof of Lemma~\ref{Thm2.4} is given in
\cite[Thm.1, pp.353--359]{arithMertens}.
Trivially, it implies Dirichlet's prime number theorem
on the infinity of primes in an arithmetic progression $b$ mod $a$
with $(a,b)=1$.

\medskip
Another handy Lemma that we use in Section~\ref{sec5}
is from Tenenbaum's book \cite[III.5.2,Thm.3,p.366]{Tenenbaum}.
\begin{lemma}
  \label{LemTen}
  For $k\in\N$, $z\in\R_{\geq 0}$, $a_1,\dots,a_k\in\R_{>0}$,
  let 
  $N_{k}(z):=\#\{(\nu_{1},\dots,\nu_{k})\in\N_{0}^{k};\
    \sum_{i=1}^{k}\nu_{i}a_{i}\leq z\}$.  Then
    \[
      \frac{z^{k}}{k!}\prod_{i=1}^{k}\frac{1}{a_{i}}<N_{k}(z)
      <\frac{(z+\sum_{i=1}^{k}a_{i})^{k}}{k!}\prod_{i=1}^{k}\frac{1}{a_{i}}.
    \]
\end{lemma}

\medskip
Finally, the following well-known elementary property
of the greatest common divisor
is frequently used in our work.
\begin{lemma}
\label{Lem2.5} 
  Let $a,b,c\in\Z$ be integers such that $a\equiv b\mod c$. Then $(a,c)=(b,c)$.
\end{lemma}


\section{The function $\frac{\phi(an+b)}{cn+d}$}
\label{sec3}

In this section, we present a proof of
Theorem~\ref{Thm3.1} in Section~\ref{sec1}.

\begin{proof}[Proof of Theorem~\ref{Thm3.1}]
The inclusion $L\subseteq\ ]0,D]$ is assured by the
definition of $M$. Let $x\in\ ]0,D[$ be an arbitrary 
interval point and put $A:=x^{\phi(a)}$ and $B:=D^{\phi(a)}$. Then $0<A<B$. 

Let $a':=a/(a,b)$ and $b':=b/(a,b)$, so that $(a',b')=1$.
Let $(p_k^*)_{k\in\N}$ be the increasing sequence of all primes such
that $p_k^*\equiv b'\mod a'$, of which there are infinitely many by
Dirichlet's prime number theorem, and let $(r_k)_{k\in\N}$ be the
increasing sequence of integers such that $p_k^*=a'r_k+b'$.
Then $(a,b)p_k^*=ar_k+b$. (Note that if  $b=0$,
then $a'=1$ and $p_k^*$ is the sequence of all primes.)

Now define
\[
m_k:=(a,b)p_{f(k)}^*\prod_{\substack{e^{Ak}<p\leq e^{Bk}\\p\equiv 1\mod a}} p
\]
for a monotonically increasing function $f:\N\to\N$ 
such that $p_{f(k)}^*>e^{Bk}$ for all $k\in\N$.
Since $(a,b)p_{f(k)}^*=ar_{f(k)}+b$,
we have $m_k\equiv b\mod a$, which means that 
there is an $n_k\in\Z$ such that
$m_k=an_k+b$. Since $p_{f(k)}^*\to\infty$ for $k\to\infty$, also
$m_k\to\infty$ and $n_k\to\infty$. Hence 
we have $n_k\in\N$ and $n_k>\max\{-b/a,-d/c\}$ for sufficiently large $k$. 
Let $k$ also be sufficiently large such that $p\leq e^{Ak}$
for all $p\mid(a,b)t$. Then $(t,m_k/(a,b))=1$
and the numbers $(a,b)$ and $p_{f(k)}^*$ and
$\displaystyle\prod_{\substack{e^{Ak}<p\leq e^{Bk}\\p\equiv 1\mod a}} p$ are pairwise
coprime. 

Since $1-\frac{1}{p_{f(k)}^*}\to1$ for $k\to\infty$,
Lemma~\ref{Thm2.4} yields
\begin{align*}
\frac{\phi(m_k)}{m_k}
&=\frac{\phi((a,b))}{(a,b)}\cdot
\left(1-\frac{1}{p_{f(k)}^*}\right)\cdot
\prod_{\substack{e^{Ak}<p\leq e^{Bk}\\p\equiv 1\mod a}}\left(1-\frac{1}{p}\right)\\
&\sim\frac{\phi((a,b))}{(a,b)}\cdot
\left(\frac{\log(e^{Ak})}{\log(e^{Bk})}\right)^{1/\phi(a)}\\
&=\frac{\phi((a,b))}{(a,b)}\cdot
\left(\frac{A}{B}\right)^{1/\phi(a)}
=\frac{\phi((a,b))}{(a,b)}\cdot
\frac{x}{D}.
\end{align*}
From $n_k\to\infty$ for $k\to\infty$ and $m_k=an_k+b$, we obtain
\begin{align*}
\frac{\phi(an_k+b)}{cn_k+d}
&=\frac{\phi(m_k)}{m_k}\cdot\frac{an_k+b}{cn_k+d}
=\frac{\phi(m_k)}{m_k}\cdot\frac{a}{c}\cdot\frac{n_k+b/a}{n_k+d/c}\\
&\sim\frac{\phi((a,b))}{(a,b)}\cdot
\frac{x}{D}\cdot
\frac{a}{c}
=D\cdot
\frac{x}{D}=x.
\end{align*}
Since $x<D$, we deduce $n_k\in M$ for sufficiently large $k$,
thus $x\in\ ]0,D[$ is the limit of a sequence from $M$. Obviously, any
of the properties (1),(2),(3) is true for $n_k$, which gives the statement.
\end{proof}

Theorem~\ref{Thm3.1} being stated with fairly general
constants $a,b,c,d$ can also be used
to deduce immediately a version with
$n$ in $L$ being restricted to an arithmetic progression $r$ mod $m$.

\begin{cor}\label{Cor3.3}
  Let $a,m\in\N$, $b,r\in\Z$, $c\in\R_{>0}$, $d\in\R$ be constants. Put
\[D:=\frac{a}{c}\cdot\frac{\phi((am,ar+b))}{(am,ar+b)}\]
and
\[M:=\left\{ n\in\N;\ n> \max\Big\{-\frac{b}{a},-\frac{d}{c}\Big\}
\text{ and } \frac{\phi(an+b)}{cn+d}\leq D\right\}.\]
Then the set
\[L:=\left\{
\frac{\phi(an+b)}{cn+d};\ 
n\in M,\ n\equiv r\mod m
\right\}\]
is dense in $]0,D]$.
    Furthermore, for arbitrary $t\in\N$, we can restrict $L$ by any of the
    following conditions on $n$ such that $L$ remains dense in $]0,D]$.
\begin{enumerate}%
\item $\frac{an+b}{(am,ar+b)}$ is squarefree,
\item $\frac{an+b}{(am,ar+b)}$ and $t$ are coprime,
\item there is at most one prime factor $p\mid\frac{an+b}{(am,ar+b)}$
  such that $p\not\equiv1\mod a$, and if it exists, it has
  multiplicity 1 and $a>1$.
\end{enumerate}
\end{cor}

\begin{proof} The inclusion $L\subseteq\ ]0,D]$ is assured by the definition of
      $M$. Let $x\in\ ]0,D[$ be an interval point
          and let $(s_k)_{k\in\N}\subseteq\N$ be a
          sequence such that
\[
\lim_{k\to\infty}\frac{\phi(ams_k+ar+b)}{cms_k+cr+d}=x,
\]
which exists due to Theorem~\ref{Thm3.1}.

Put $n_k:=ms_k+r$. Then $n_k\equiv r\mod m$ and
\[
 \lim_{k\to\infty}\frac{\phi(an_k+b)}{cn_k+d}
=\lim_{k\to\infty}\frac{\phi(ams_k+ar+b)}{cms_k+cr+d}
=x.
\]
Since $0<x<D$, it follows that $n_k\in M$ for all sufficiently large $k$.
Furthermore,
\[
\frac{an_k+b}{(am,ar+b)}
=\frac{ams_k+ar+b}{(am,ar+b)},
\]
and the three additional conditions on $\frac{an_k+b}{(am,ar+b)}$
follow directly from the choice of $s_k$ from Theorem~\ref{Thm3.1}.
\end{proof}

\begin{rem}\label{note1}
For all $\theta\in\Z$, we have
\[
a(m\theta+r)+b
=am\theta+ar+b
\equiv ar+b\mod am,
\]
thus by Lemma~\ref{Lem2.5}, %
the term $(am,a(m\theta+r)+b)=(am,ar+b)$
appearing in the denominator of Corollary~\ref{Cor3.3}(1),(2),(3)
is independent of the choice
of the representative $r$ modulo $m$.
Thus adding some of these restrictions to $L$ is \emph{not} depending
on this representative.
\end{rem}

\begin{rem}\label{note2}
Since $ar+b\equiv b\mod a$, Lemma~\ref{Lem2.5} for $m=1$ yields 
\[(am,ar+b)=(a,ar+b)=(a,b).\] 
Hence, Theorem~\ref{Thm3.1} is the special case of
Corollary~\ref{Cor3.3} with $m=1$.
\end{rem}

\begin{rem}\label{note3}
  If $r=0$, the congruence $n\equiv r\mod m$ is equivalent to $m\mid n$,
  which means that an arbitrary divisor $m$ of $n$ is considered.
  If $r=1$, the congruence $n\equiv r\mod m$ implies that
  $n$ and $m$ are coprime, which rules out a finite set of
  divisors of $n$.
  It is also possible to combine these two cases:
  Let $m_1,m_2\in\N$ be coprime, then 
  there exist integers $x,y\in\Z$ such that $xm_1+ym_2=1$
  by B\'{e}zout's identity. Now assume that
  \[n\equiv xm_1\mod m_1m_2.\]
  
  Then $n\equiv0\mod m_1$ and $n\equiv1\mod m_2$, which means
  $m_1\mid n$ and $(m_2,n)=1$. More generally, since the
  Chinese remainder theorem gives a tool to combine a finite set
  of congruence relations into a single one, this can then be fed
  into Theorem~\ref{Thm3.1} to gain a generalization to systems
  of congruences. We do not give the details.
\end{rem}

\begin{rem}\label{Note4} For $c>0$, the condition $n>-d/c$ in $M$ rules
  out exactly those cases with $cn+d\leq0$, i.e.\ the negative values
  of the fraction $\phi(an+b)/(cn+d)$. Since $n\in\N$,
  there are only finitely many values of this fraction
  in the interval $]-\infty,0]$. However, the same is not true for
  the interval $]D,\infty[$. If we let $b=0$, $a\geq2$ , $d<0$ and choose $n$
  such that $\rad(n)\mid a$ and $n>-d/c$, of which there obviously are
  infinitely many, then Lemma~\ref{Thm2.3}(ii) yields
\[
\frac{\phi(an)}{cn+d}
=\frac{n\phi(a)}{cn+d}
=\frac{\phi(a)}{c+d/n}
>\frac{\phi(a)}{c}=D.
\]
However, since the left hand side converges to $D$ for $n\to\infty$,
such $n$ with $\rad(n)\mid a$ generate only finitely many values
of $\phi(an)/(cn+d)$ in the interval $]D+\epsilon,\infty[$ for any
$\epsilon>0$. It could still be possible that
some $x>D$ is a limit point of the set of all fractions.
But this is not true due to Corollary~\ref{Thm3.4}.
\end{rem}

\begin{cor}\label{Thm3.4}
For all constants $a\in\N$, $b\in\Z$, $c\in\R_{>0}$, $d\in\R$,
\[
\limsup_{n\to\infty}\frac{\phi(an+b)}{cn+d}
=\frac{a\phi((a,b))}{c(a,b)}
\]
and
\[
\liminf_{n\to\infty}\frac{\phi(an+b)}{cn+d}
=0.
\]
\end{cor}

\begin{proof}
  The assertion on the $\liminf$ is clear by Theorem~\ref{Thm3.1},
  since the fraction $\frac{\phi(an+b)}{cn+d}$ is negative
  only finitely many times, namely if $n<-d/c$.

Let $g:=(a,b)$. Theorem~\ref{Thm3.1} then yields
\[\limsup_{n\to\infty}\frac{\phi(an+b)}{an+b}\geq\frac{\phi(g)}{g},\]
since the sequence of fractions has $D=\phi(g)/g$
as a limit point.

Let $r_n:=na/g+b/g$. Then $r_ng=an+b$, and thus by Lemma~\ref{Thm2.3}(iii),
\[
\frac{\phi(an+b)}{an+b}
=\frac{\gamma_g(r_n)\phi(\tau_g(r_n))\phi(g)}{r_ng}
=\frac{\phi(\tau_g(r_n))}{\tau_g(r_n)}\cdot\frac{\phi(g)}{g}
\leq\frac{\phi(g)}{g},
\]
hence
\[\limsup_{n\to\infty}\frac{\phi(an+b)}{an+b}=\frac{\phi((a,b))}{(a,b)}.\]
This implies
\[
\limsup_{n\to\infty}\frac{\phi(an+b)}{cn+d}
=\limsup_{n\to\infty}\frac{\phi(an+b)}{an+b}\cdot\frac{an+b}{cn+d}
=\frac{\phi((a,b))}{(a,b)}\cdot\frac{a}{c}.
\]\end{proof}

\begin{rem}
Note that for $c<0$ we obtain immediately 
\[
\limsup_{n\to\infty}\frac{\phi(an+b)}{cn+d}
=-\liminf_{n\to\infty}\frac{\phi(an+b)}{-cn-d}
=0
\]
and
\[
\liminf_{n\to\infty}\frac{\phi(an+b)}{cn+d}
=-\limsup_{n\to\infty}\frac{\phi(an+b)}{-cn-d}
=\frac{a\phi((a,b))}{c(a,b)}.
\]
\end{rem}

\section{The fraction $\frac{\phi(an+b)}{\phi(cn+d)}$}
\label{sec4}

In this section we investigate the fraction $\phi(an+b)/\phi(cn+d)$
for all constants $a,c\in\N$, $b,d\in\Z$. Clearly, it takes only
positive values,
and we prove here our next main result, Theorem~\ref{Thm4.2},
which states that the set of fractions with $ad\neq bc$ 
is dense in the positive real numbers, whereas for $ad=bc$, there
are exactly $2^{\omega(\tau_c(a)\tau_a(c))}$ such fractions.

\medskip
We start the proof by treating first the case $a=c=1,b\neq d$, with $n$ in an arithmetic progression.

\begin{lemma}\label{Lem4.1}
For all $m\in\N$ and $b,d,r\in\Z$ such that $b\neq d$, the set
\[\left\{
\frac{\phi(n+b)}{\phi(n+d)};\ 
n\in\N,\ n>\max\{-b,-d\},\ n\equiv r\mod m
\right\}\]
is dense in $]0,\infty[$.
\end{lemma}

The following proof is adapted and expanded from Schinzel's work
\cite[Thm.~1, pp.~123--128]{schinzel}.

\begin{proof}
  Let $z>0$ be a positive real number which we consider as an
  interval point of $]0,\infty[$. It suffices to construct %
a sequence $(n_k)_{k\in\N}\subseteq\N$ such that $n_k\equiv r\mod m$
and $n_k>\max\{-b,-d\}$ for sufficiently large $k$ and
\[
\lim_{k\to\infty}\frac{\phi(n_k+b)}{\phi(n_k+d)}=z.
\]
We assume without loss of generality that $r+b>0$ and $r+d>0$,
since otherwise we replace $r$ by $r+\theta m$ for any sufficiently
large $\theta\in\N$. %
Let
\begin{equation}\label{eq1}
  z':=\frac{\phi(r+d)/(r+d)}{\phi(r+b)/(r+b)}\cdot z>0,
\end{equation}
and let $h\in\N_0$ be such that $p_h$ is the largest prime factor of
\[m\cdot|b-d|\cdot(r+b)(r+d)>0,\]
respectively let $h=0$ if this integer equals $1$.
Thus we define
\[
F:=m\cdot|b-d|\cdot(r+b)(r+d)\cdot p_1\cdots p_h>0,
\]
and note that $\rad(F)=p_1\cdots p_h$.

\bigskip
\textsc{Step 1}: \emph{Construct two sequences $A_k,B_k\in\N$ 
  with $(A_k,B_k)=1$ and a simple prime factorization such that}
\[
\lim_{k\to\infty}\frac{\phi(B_k)/B_k}{\phi(A_k)/A_k}=z'.
\]
Let $k\in\N$ and put $A_k:=p_{h+1}\cdots p_{h+k}$. Then Mertens' theorem,
here Lemma~\ref{Thm2.4} with $a=1$, yields
\begin{equation}\label{eq2}
\lim_{k\to\infty}\frac{\phi(A_k)}{A_k}
=\lim_{k\to\infty}\prod_{i=h+1}^{h+k}\frac{p_i-1}{p_i}=0.
\end{equation}
Since $\prod_{i=h+k+1}^\infty\frac{p_i-1}{p_i}=0$, there is a minimal
$\ell(k)\in\N_0$ %
such that
\[
\prod_{i=h+k+1}^{h+k+\ell(k)}\frac{p_i-1}{p_i}
<z'\cdot\frac{\phi(A_k)}{A_k}\in\ ]0,\infty[.
\]
Using this, put $B_k:=p_{h+k+1}\cdots p_{h+k+\ell(k)}$
if $\ell(k)\geq 1$,
and $B_k:=1$ if $\ell(k)=0$.

Then $F,A_k,B_k$ are
pairwise coprime and $\rad(F)A_kB_k=p_1\cdots p_{h+k+\ell(k)}$. Furthermore,
\begin{equation}\label{eq3}
  \frac{\phi(B_k)/B_k}{\phi(A_k)/A_k}<z'.\end{equation}

If $\ell(k)=0$, then it follows from $B_k=1$ and \eqref{eq3} that
\[
   \frac{1}{\phi(A_k)/A_k}<z',
\]
but the denominator converges to 0 by \eqref{eq2}, hence there is
an $N\in\N$ such that for all $k\geq N$ we have $\ell(k)>0$.

From now on
we assume that $k\geq N$ %
with $B_k\geq2$. Then the minimality of $\ell(k)\geq 1$ gives us
\[
\frac{p_{h+k+\ell(k)}}{p_{h+k+\ell(k)}-1}
\cdot\frac{\phi(B_k)}{B_k}
=\prod_{i=h+k+1}^{h+k+\ell(k)-1}\frac{p_i-1}{p_i}
\geq z'\cdot\frac{\phi(A_k)}{A_k},
\]
hence
\begin{equation}\label{eq4}
  \frac{p_{h+k+\ell(k)}}{p_{h+k+\ell(k)}-1}
  \cdot\frac{\phi(B_k)/B_k}{\phi(A_k)/A_k}\geq z'.
\end{equation}
Since
\[
\lim_{k\to\infty}\frac{p_{h+k+\ell(k)}}{p_{h+k+\ell(k)}-1}
=\lim_{k\to\infty}\left(1-\frac{1}{p_{h+k+\ell(k)}}\right)^{-1}=1,
\]
it follows from \eqref{eq3} and \eqref{eq4}, that
\begin{equation}\label{eq5}
  \lim_{k\to\infty}\frac{\phi(B_k)/B_k}{\phi(A_k)/A_k}=z'.
\end{equation}

Note that for $k\to\infty$ we have $A_k\geq p_{h+k}\to\infty$
and $B_k\geq p_{h+k+\ell(k)}\to\infty$, since $\ell(k)\geq1$.

\bigskip
\textsc{Step 2}: \emph{Construct a sequence $n_k\in\N$ such that
  $n_k\equiv r\mod m$ and numbers $x_k,y_k\in\N$ such that
  $n_k+d=A_ky_k$ and $n_k+b=B_kx_k$ for sufficiently large $k$.}

Since $F,A_k$ and $B_k$ are pairwise coprime by definition,
it follows from
the Chinese remainder theorem that there exists an $n_k\in\N$ such that
$FA_kB_k<n_k\leq 2FA_kB_k$ and
\begin{align}
  n_k&\equiv  r \mod F,
  \label{eq6} \\
n_k&\equiv -d \mod A_k,\notag\\
n_k&\equiv -b \mod B_k.\notag
\end{align}
By the first congruence \eqref{eq6}, there is an integer
$\sigma_k\in\Z$ such that $n_k=\sigma_kF+r$. Since $m\mid F$,
it follows that $n_k\equiv r\mod m$.
Furthermore, there are integers $x_k,y_k\in\Z$ such that
\begin{equation}\label{eq7}
  n_k+d=A_ky_k,\ n_k+b=B_kx_k.
\end{equation}
Let $k$ be sufficiently large such that $FA_kB_k>\max\{-b,-d\}$,
thus $n_k>\max\{-b,-d\}$, which means $x_k,y_k\in\N$.

\bigskip
\textsc{Step 3}: \emph{Show that $x_k$ and $y_k$ disappear
  asymptotically if inserted into \eqref{eq5}, %
  i.e. show that}
\[
\frac{\phi(B_kx_k)/(B_kx_k)}{\phi(A_ky_k)/(A_ky_k)}
\sim
\frac{\phi(\gamma_F(x_k))/\gamma_F(x_k)}
          {\phi(\gamma_F(y_k))/\gamma_F(y_k)}\cdot
\frac{\phi(B_k)/B_k}{\phi(A_k)/A_k},
\] 
\emph{where the independence from $x_k,y_k$ of this fraction
  follows from} \textsc{Step~4}.

Assume $y_k>A_kB_k$. Then
\[
A_k^2B_k<A_ky_k=n_k+d\leq 2FA_kB_k+d,
\]
hence
\[
A_kB_k(A_k-2F)<d.
\]
Thus
let $k$ be large enough such that $A_kB_k(A_k-2F)\geq d$, so
that $y_k\leq A_kB_k$.

Analogously, 
let $k$ be large enough such that $x_k\leq A_kB_k$.

Note that
\[
A_ky_k-B_kx_k=d-b.
\]
If there was a prime $p$ such that $p\mid(x_k,A_k)$ or $p\mid(y_k,B_k)$,
then $p\mid d-b$, which would contradict $(F,A_k)=(F,B_k)=1$
since $d-b\mid F$. Hence
\[
(x_k,A_k)=(y_k,B_k)=1.
\]
Now let $p$ be a prime such that $p\mid y_k$ and $p\nmid A_k$.
Since $(y_k,B_k)=1$, it follows that $p>p_{h+k+\ell(k)}$ or $p\leq p_h$.
For $n_k+d$ we deduce the factorization
\begin{equation}\label{eq8}
n_k+d=A_ky_k
=\gamma_F(y_k)
\cdot p_{h+1}^{\alpha_1}\cdots p_{h+k}^{\alpha_k}
\cdot u_1^{\beta_1}\cdots u_{s(k)}^{\beta_{s(k)}} 
\end{equation}
with $s(k)\in\N_0$, with
$\alpha_1,\dots,\alpha_k,\beta_1,\dots,\beta_{s(k)}\in\N$ and
primes $u_1,\dots,u_{s(k)}$ such that
\begin{equation}\label{eq9}
  p_{h+k+\ell(k)}<u_1<\cdots<u_{s(k)}.
\end{equation}
Note that a factorization of $y_k$ is given by
\[
y_k=\gamma_F(y_k)
\cdot p_{h+1}^{\alpha_1-1}\cdots p_{h+k}^{\alpha_k-1}
\cdot u_1^{\beta_1}\cdots u_{s(k)}^{\beta_{s(k)}}.
\]
Now assume that $s(k)\geq k+\ell(k)$. Then by \eqref{eq9},
\[
y_k
\geq u_1\cdots u_{s(k)}
>p_{h+k+\ell(k)}^{s(k)}
\geq p_{h+k+\ell(k)}^{k+\ell(k)}
\geq A_kB_k,
\]
which contradicts $y_k\leq A_kB_k$, hence $s(k)<k+\ell(k)$.

Furthermore, $u_i\geq p_{h+k+\ell(k)}+i$ for all $1\leq i\leq s(k)$
by \eqref{eq9}, hence
\begin{align}
1\geq\prod_{i=1}^{s(k)}\left(1-\frac{1}{u_i}\right)
&\geq\prod_{i=1}^{s(k)}\left(1-\frac{1}{p_{h+k+\ell(k)}+i}\right)\notag\\
&=   \prod_{i=1}^{s(k)}\frac{p_{h+k+\ell(k)}+i-1}{p_{h+k+\ell(k)}+i}
=\frac{p_{h+k+\ell(k)}}{p_{h+k+\ell(k)}+s(k)}\notag\\
&>\frac{p_{h+k+\ell(k)}}{p_{h+k+\ell(k)}+k+\ell(k)},
\label{eq10}
\end{align}
since the last product is a telescoping product.

From the prime number theorem we know that $p_k\sim k\log(k)$
for $k\to\infty$, and thus
\[
\frac{k+\ell(k)}{p_{h+k+\ell(k)}}
\sim\frac{k+\ell(k)}{(h+k+\ell(k))\log(h+k+\ell(k))}\to 0.
\]
Hence
\begin{equation}\label{eq11}
\lim_{k\to\infty}\frac{p_{h+k+\ell(k)}}{p_{h+k+\ell(k)}+k+\ell(k)}
=\lim_{k\to\infty}\left(1+\frac{k+\ell(k)}{p_{h+k+\ell(k)}}\right)^{-1}=1.
\end{equation}
Using \eqref{eq10} and \eqref{eq11}, we deduce that 
\begin{equation}\label{eq12}
\prod_{i=1}^{s(k)}\left(1-\frac{1}{u_i}\right)\sim1\end{equation}
for $k\to\infty$. With \eqref{eq8} we obtain
\[
\frac{\phi(n_k+d)}{n_k+d}
=\frac{\phi(\gamma_F(y_k))}{\gamma_F(y_k)}\cdot
 \frac{\phi(A_k)}{A_k}\cdot
 \prod_{i=1}^{s(k)}\left(1-\frac{1}{u_i}\right).
\]
Then \eqref{eq12} yields  
\begin{equation}\label{eq13}
\frac{\phi(n_k+d)}{n_k+d}
\sim
\frac{\phi(\gamma_F(y_k))}{\gamma_F(y_k)}\cdot
\frac{\phi(A_k)}{A_k}
\end{equation}
for $k\to\infty$.

Analogously, we obtain by $(x_k,A_k)=1$ the factorization
\begin{equation}\label{eq14}
n_k+b=B_kx_k
=\gamma_F(x_k)
\cdot p_{h+k+1}^{\delta_{1}}\cdots p_{h+k+\ell(k)}^{\delta_{\ell(k)}}
\cdot v_1^{\epsilon_1}\cdots v_{t(k)}^{\epsilon_{t(k)}}
\end{equation}
with ${t(k)}\in\N_0$, with $\delta_1,\dots,\delta_{\ell(k)},
\epsilon_1,\dots,\epsilon_{t(k)}\in\N$ and primes $v_1,\dots,v_{t(k)}$ such that
\[
p_{h+k+\ell(k)}<v_1<\cdots<v_{t(k)}.
\]
Assume that ${t(k)}\geq k+\ell(k)$. Then
\[
x_k
\geq v_1\cdots v_{t(k)}
>p_{h+k+\ell(k)}^{t(k)}
\geq p_{h+k+\ell(k)}^{k+\ell(k)}
\geq A_kB_k,
\]
which contradicts $x_k\leq A_kB_k$, hence $t(k)<k+\ell(k)$.

Again, we have $v_i\geq p_{h+k+\ell(k)}+i$ for all $1\leq i\leq t(k)$ and thus
\begin{align}
1\geq\prod_{i=1}^{t(k)}\left(1-\frac{1}{v_i}\right)
&\geq\prod_{i=1}^{t(k)}\left(1-\frac{1}{p_{h+k+\ell(k)}+i}\right) \notag \\
&=   \prod_{i=1}^{t(k)}\frac{p_{h+k+\ell(k)}+i-1}{p_{h+k+\ell(k)}+i}
=\frac{p_{h+k+\ell(k)}}{p_{h+k+\ell(k)}+t(k)} \notag \\
&>\frac{p_{h+k+\ell(k)}}{p_{h+k+\ell(k)}+k+\ell(k)}.
\label{eq15}
\end{align}
Hence by \eqref{eq15} and \eqref{eq11}, we obtain 
\begin{equation}
  \label{eq16}
  \prod_{i=1}^{t(k)}\left(1-\frac{1}{v_i}\right)\sim1
\end{equation}
for $k\to\infty$. With \eqref{eq14} we obtain
\begin{align*}
\frac{\phi(n_k+b)}{n_k+b}
=\frac{\phi(\gamma_F(x_k))}{\gamma_F(x_k)}\cdot
 \frac{\phi(B_k)}{B_k}\cdot
 \prod_{i=1}^{t(k)}\left(1-\frac{1}{v_i}\right).
\end{align*}
Hence \eqref{eq16} shows 
\begin{equation}\label{eq17}
\frac{\phi(n_k+b)}{n_k+b}
\sim
\frac{\phi(\gamma_F(x_k))}{\gamma_F(x_k)}\cdot
\frac{\phi(B_k)}{B_k}
\end{equation}
for $k\to\infty$.
Since $\lim_{k\to\infty}\frac{n_k+d}{n_k+b}=1$,
an application of Equation~\eqref{eq17}, \eqref{eq13} and \eqref{eq5}
shows that
\begin{align}
     \frac{\phi(n_k+b)}{\phi(n_k+d)}
&\sim\frac{\phi(n_k+b)}{n_k+b}\cdot
     \frac{n_k+d}{\phi(n_k+d)} \notag\\
&\sim\frac{\phi(\gamma_F(x_k))/\gamma_F(x_k)}
          {\phi(\gamma_F(y_k))/\gamma_F(y_k)}\cdot
     \frac{\phi(B_k)/B_k}{\phi(A_k)/A_k} \notag\\
&\sim\frac{\phi(\gamma_F(x_k))/\gamma_F(x_k)}
          {\phi(\gamma_F(y_k))/\gamma_F(y_k)}\cdot z'
          \label{eq18}
\end{align}
for $k\to\infty$.

\textsc{Step 4}: \emph{The factor $\displaystyle
  \frac{\phi(\gamma_F(x_k))/\gamma_F(x_k)}{\phi(\gamma_F(y_k))/\gamma_F(y_k)}$
  is constant.}

By \eqref{eq7} and \eqref{eq6}, we have
\[A_ky_k=n_k+d\equiv r+d\mod F.\]
Hence $(A_ky_k,F)=(r+d,F)=r+d$ by Lemma~\ref{Lem2.5} since $r+d\mid F$.
By $(F,A_k)=1$ it follows that $(y_k,F)
=r+d$, thus
\begin{equation}\label{eq19}
\frac{\phi(\gamma_F(y_k))}{\gamma_F(y_k)}
=\prod_{\substack{p\mid y_k\\p\mid F}}\left(1-\frac{1}{p}\right)
=\prod_{p\mid(y_k,F)}\left(1-\frac{1}{p}\right)
=\frac{\phi(r+d)}{r+d}.
\end{equation}
Analogously, we obtain by \eqref{eq7} and \eqref{eq6}
\[B_kx_k=n_k+b\equiv r+b\mod F\]
and thus $(B_kx_k,F)=(r+b,F)=r+b$ by Lemma~\ref{Lem2.5}
since $r+b\mid F$. By $(F,B_k)=1$ it follows that
$(x_k,F)
=r+b$, thus
\begin{equation}\label{eq20}
\frac{\phi(\gamma_F(x_k))}{\gamma_F(x_k)}
=\prod_{\substack{p\mid x_k\\p\mid F}}\left(1-\frac{1}{p}\right)
=\prod_{p\mid(x_k,F)}\left(1-\frac{1}{p}\right)
=\frac{\phi(r+b)}{r+b}.
\end{equation}
Now Equations~\eqref{eq18},\eqref{eq19}, \eqref{eq20} and \eqref{eq1}
show that finally
\[
\frac{\phi(n_k+b)}{\phi(n_k+d)}
\sim\frac{\phi(\gamma_F(x_k))/\gamma_F(x_k)}
         {\phi(\gamma_F(y_k))/\gamma_F(y_k)}\cdot z'
=\frac{\phi(r+b)/(r+b)} %
      {\phi(r+d)/(r+d)}\cdot z'
=z
\]
for $k\to\infty$, which completes the proof.\end{proof}

Having proved this special case, we use the result of Lemma~\ref{Lem4.1}
to prove the general case hereinafter, which is the first statement of Theorem~\ref{Thm4.2}

\begin{theorem}\label{Thm4.2 dense case}
For all $a,c\in\N$ and $b,d\in\Z$ such that $ad\neq bc$, the set
\[\left\{
\frac{\phi(an+b)}{\phi(cn+d)};\ 
n\in\N,\ 
n>\max\Big\{-\frac{b}{a}, -\frac{d}{c}\Big\}
\right\}\]
is dense in $]0,\infty[$.
\end{theorem}

\begin{proof}
  Let $\alpha_n:=an+b$ and $\beta_n:=cn+d$ and let $r\in\N$ be large enough such that $ar+b>0$ and $cr+d>0$.
  By Lemma~\ref{Thm2.3}(iii), 
\[
\frac{\phi(a\beta_n)}{\phi(c\alpha_n)}
=\frac{\gamma_{\beta_n}(a)\phi(\tau_{\beta_n}(a))}
      {\gamma_{\alpha_n}(c)\phi(\tau_{\alpha_n}(c))}\cdot
 \frac{\phi(\beta_n)}{\phi(\alpha_n)},
\]
hence
\begin{align}
\frac{\phi(an+b)}{\phi(cn+d)}
&=\frac{\phi(\alpha_n)}{\phi(\beta_n)} 
=\frac{\gamma_{\beta_n}(a)\phi(\tau_{\beta_n}(a))}
       {\gamma_{\alpha_n}(c)\phi(\tau_{\alpha_n}(c))}\cdot
  \frac{\phi(c\alpha_n)}{\phi(a\beta_n)}\notag\\
&=\frac{\gamma_{\beta_n}(a)\phi(\tau_{\beta_n}(a))}
       {\gamma_{\alpha_n}(c)\phi(\tau_{\alpha_n}(c))}\cdot
       \frac{\phi(acn+bc)}{\phi(acn+ad)}. 
       \label{eqq1}
\end{align}
Let $z>0$ be a real constant
and let $(s_k)_{k\in\N}\subseteq\N$ be a sequence from Lemma~\ref{Lem4.1}
such that $s_k>\max\{-bc,-ad\}$ and $s_k\equiv acr\mod a^2c^2$
for all $k$ with
\begin{equation}\label{eqq2}
\lim_{k\to\infty}\frac{\phi(s_k+bc)}{\phi(s_k+ad)}
=\frac{\gamma_{\alpha_r}(c)\phi(\tau_{\alpha_r}(c))}
      {\gamma_{\beta _r}(a)\phi(\tau_{\beta _r}(a))} \cdot z>0.
\end{equation}
Then there exists a sequence $x_k\in\Z$ such that
$s_k=x_ka^2c^2+acr$. Define $n_k:=s_k/(ac)$. %
Then $n_k>\max\{-b/a,-d/c\}$ and $n_k=x_k ac+r$ for all $k$,
which implies $n_k\equiv r\mod a,c$. Now
\[\alpha_{n_k}=an_k+b\equiv ar+b\mod c.\]
Thus by Lemma~\ref{Lem2.5} we know that
$(\alpha_{n_k},c)=(\alpha_r,c)$ is constant.
Hence $\gamma_{\alpha_{n_k}}(c)=\gamma_{\alpha_r}(c)$ and
\[
\tau_{\alpha_{n_k}}(c)
=\frac{c}{\gamma_{\alpha_{n_k}}(c)}
=\frac{c}{\gamma_{\alpha_r}(c)}
=\tau_{\alpha_r}(c).
\]
Analogously,
\[\beta_{n_k}=cn_k+d\equiv cr+d\mod a.\]
Thus by Lemma~\ref{Lem2.5}, $(\beta_{n_k},a)=(\beta_r,a)$ is also
constant. Then $\gamma_{\beta_{n_k}}(a)=\gamma_{\beta_r}(a)$,
and $\tau_{\beta_{n_k}}(a)=\tau_{\beta_r}(a)$. Therefore the fraction
\[
\frac{\gamma_{\beta _{n_k}}(a)\phi(\tau_{\beta _{n_k}}(a))}
     {\gamma_{\alpha_{n_k}}(c)\phi(\tau_{\alpha_{n_k}}(c))}
=
\frac{\gamma_{\beta _r}(a)\phi(\tau_{\beta _r}(a))}
     {\gamma_{\alpha_r}(c)\phi(\tau_{\alpha_r}(c))}
\]
is constant, too. Using \eqref{eqq1} and \eqref{eqq2}, we obtain
\begin{align*}
\frac{\phi(an_k+b)}{\phi(cn_k+d)}
&=\frac{\gamma_{\beta _r}(a)\phi(\tau_{\beta _r}(a))}
       {\gamma_{\alpha_r}(c)\phi(\tau_{\alpha_r}(c))}\cdot
  \frac{\phi(acn_k+bc)}{\phi(acn_k+ad)}\\
&=\frac{\gamma_{\beta _r}(a)\phi(\tau_{\beta _r}(a))}
       {\gamma_{\alpha_r}(c)\phi(\tau_{\alpha_r}(c))}\cdot
  \frac{\phi(s_k+bc)}{\phi(s_k+ad)}\to z
\end{align*}
for $k\to\infty$.
\end{proof}

To prove the second statement of Theorem~\ref{Thm4.2},
we require an auxiliary result.

\begin{lemma}\label{Lemma fuer Zaehlen von phi Werten}
\begin{enumerate}[(i)]
\item
Let $a,c,n\in\N$ and $b,d\in\Z$. Then
\[\tau_{cn+d}(\gamma_c(a))=\tau_d(\gamma_c(a)),
\quad
\tau_{an+b}(\gamma_a(c))=\tau_b(\gamma_a(c)).\]
\item
  Let $t_1,t_2,u_1,u_2\in\N$ be such that $(t_i,u_j)=1$ for
  $i,j=1,2$, and $\rad(t_1)\neq\rad(t_2)$ or $\rad(u_1)\neq\rad(u_2)$. Then
\begin{equation}\label{lemma fuer Zaehlen ii}
\frac{\phi(t_1)/t_1}{\phi(u_1)/u_1}
\neq
\frac{\phi(t_2)/t_2}{\phi(u_2)/u_2}.
\end{equation}
\item
For all $n\in\N$,
\[\#\{t\in\N;\ t\Mid n\}=2^{\omega(n)}.\]
\end{enumerate}
\end{lemma}
\begin{proof}
\textsc{Assertion }i):
Let $p^k\Mid\tau_{cn+d}(\gamma_c(a))$. Then $p\nmid cn+d$ and $p\mid c$,
so $p\nmid d$ and thus $p^k\Mid\tau_d(\gamma_c(a))$. 
Now let $p^k\Mid\tau_d(\gamma_c(a))$. Then $p\nmid d$ and $p\mid c$,
thus $p\nmid cn+d$ and $p^k\Mid\tau_{cn+d}(\gamma_c(a))$. Therefore,
\[\tau_{cn+d}(\gamma_c(a))=\tau_d(\gamma_c(a)).\]
We prove $\tau_{an+b}(\gamma_a(c))=\tau_b(\gamma_a(c))$ in exactly the same way.

\medskip
\textsc{Assertion }ii):
We can assume without loss of generality
that $(t_1,t_2)=(u_1,u_2)=1$,
since
for any common prime factor of e.g.\ $t_1$ and $t_2$,
the resulting common factor $1-1/p$ in the product representation of
$\phi(t_1)/t_1$ and $\phi(t_2)/t_2$
can be divided out in each side of
\eqref{lemma fuer Zaehlen ii}.
Assuming the contrary of the statement, we get
\[
t_2u_1\phi(t_1)\phi(u_2)
=
t_1u_2\phi(t_2)\phi(u_1).
\]
Since $(t_2u_1,t_1u_2)=1$, it follows that $t_2u_1\mid\phi(t_2)\phi(u_1)$, hence
\[
\phi(t_2)\phi(u_1)
\geq t_2u_1
\geq \phi(t_2)u_1
\]
and thus $u_1\leq\phi(u_1)$, which implies $u_1=1$.
Similarly, we obtain $u_2=t_1=t_2=1$, which contradicts the
assumption that
$\rad(t_1)\neq \rad(t_2)$ or $\rad(u_1)\neq\rad(u_2)$.

\medskip
\textsc{Assertion }iii):
Let $n=q_1^{k_1}\cdots q_m^{k_m}$ with $m=\omega(n)$ be the prime
factorization of $n$. Then any $t\Mid n$ must be of the form
$t=\prod_{i=1}^m q_i^{e_ik_i}$
for some $e_1,...,e_m\in\{0,1\}$. We thus have 2 choices
for each of the $e_i$, which shows that the count of such
exact divisors $t$ must be equal to $2^m$.
\end{proof}

Now we can complete the proof of Theorem~\ref{Thm4.2} by proving
its second statement.

\begin{theorem}\label{Thm4.2 not dense case}
Let $a,c\in\N$ and $b,d\in\Z$ be constants such that $ad=bc$ and define
\[L:=\left\{
\frac{\phi(an+b)}{\phi(cn+d)};\ 
n\in\N,\ 
n>\max\Big\{-\frac{b}{a}, -\frac{d}{c}\Big\}
\right\}.\]
Then $\#L=2^{\omega(\tau_c(a)\tau_a(c))}$.
\end{theorem}

\begin{proof}
Let $\Phi(n):=\phi(n)/n$.
Since we assume $ad=bc$, Equation (22) and Lemma~\ref{lemma1}(i) yield
\begin{equation}\label{hauptsatz 2 entartet - phibruch zu Phibruch}
\frac{\phi(an+b)}{\phi(cn+d)}
=
\frac{\gamma_{cn+d}(a)}
     {\gamma_{an+b}(c)}
\cdot
\frac{\phi(\tau_{cn+d}(a))}
     {\phi(\tau_{an+b}(c))}
=
\frac{a}{c}\cdot
\frac{\Phi(\tau_{cn+d}(a))}
     {\Phi(\tau_{an+b}(c))}.
\end{equation}

Now we show that
\begin{align}\label{eq:phibeh}
&\left\{
\frac{\Phi(\tau_{cn+d}(a))}
     {\Phi(\tau_{an+b}(c))}
\,\middle|\,
n\in\N,\ n>\max\Big\{-\frac{b}{a},-\frac{d}{c}\Big\}
\right\} \notag
\\&=
\frac{\Phi(\tau_d(\gamma_c(a)))}
     {\Phi(\tau_b(\gamma_a(c)))}
\cdot
\left\{
\frac{\Phi(t)}{\Phi(u)}
\,\middle|\,
t,u\in\N,\ t\Mid\tau_c(a),\ u\Mid\tau_a(c)
\right\}.
\end{align}

By Lemma~\ref{lemma1}(i), we obtain
\[
\tau_{cn+d}(a)
=\tau_{cn+d}(\gamma_c(a))\cdot\tau_{cn+d}(\tau_c(a)),
\]
where $\tau_{cn+d}(\gamma_c(a))=\tau_d(\gamma_c(a))$ by
Lemma~\ref{Lemma fuer Zaehlen von phi Werten}(i)
and $t:=\tau_{cn+d}(\tau_c(a))$ is an exact divisor of $\tau_c(a)$.
Thus $\tau_{cn+d}(a)$ is of the form
\begin{equation}\label{hauptsatz 2 entartet - Form von tau a}
\tau_{cn+d}(a)
=
\tau_d(\gamma_c(a))\cdot t
\end{equation}
for a $t\Mid\tau_c(a)$. Likewise, we can see that
\[
\tau_{an+b}(c)
=\tau_{an+b}(\gamma_a(c))\cdot\tau_{an+b}(\tau_a(c)),
\]
where $\tau_{an+b}(\gamma_a(c))=\tau_b(\gamma_a(c))$ by
Lemma~\ref{Lemma fuer Zaehlen von phi Werten}(i)
and $u:=\tau_{an+b}(\tau_a(c))\Mid\tau_a(c)$.
Thus $\tau_{an+b}(c)$ is of the form
\begin{equation}\label{hauptsatz 2 entartet - Form von tau c}
\tau_{an+b}(c)
=
\tau_b(\gamma_a(c))\cdot u
\end{equation}
for a $u\Mid\tau_a(c)$. This shows the first inclusion $\subseteq$ in
\eqref{eq:phibeh}.

Now we show the second inclusion $\supseteq$, namely
that for all 
$t,u$ with $t\Mid\tau_c(a)$, $u\Mid\tau_a(c)$,
there are infinitely many $n$
that satisfy both \eqref{hauptsatz 2 entartet - Form von tau a}
and \eqref{hauptsatz 2 entartet - Form von tau c} simultaneously.
Let $\tau_c(a)=q_1^{k_1}\cdots q_m^{k_m}$ and
$\tau_a(c)=r_1^{h_1}\cdots r_\ell^{h_\ell}$ be their prime factorizations
and let $t,u\in\N$ be such that $t\Mid\tau_c(a)$ and $u\Mid\tau_a(c)$.
Then $t$ and $u$ must be of the form
\begin{equation}\label{hauptsatz 2 entartet - Form von t,u}
t=\prod_{i=1}^m q_i^{e_ik_i},
\quad
u=\prod_{i=1}^\ell r_i^{f_ih_i}
\end{equation}
for some $e_1,...,e_m,f_1,...,f_\ell\in\{0,1\}$.
Let us define $Q:=q_1\cdots q_m=\rad(\tau_c(a))$
and $R:=r_1\cdots r_\ell=\rad(\tau_a(c))$ and observe that $Q$ and $R$
are coprime. By the Chinese remainder theorem, there is an $x\in\Z$ such that
\begin{alignat}{2}
\begin{aligned}\label{hauptsatz 2 entartet - Kongruenzen x}
x&\equiv  d &&\mod c,\\
x&\equiv e_i&&\mod q_i,\text{ for }i=1,...,m
\end{aligned}
\end{alignat}
and the solutions of this congruence system are exactly
those $n\in\Z$ that satisfy $n\equiv x\mod cQ$.
Furthermore, we obtain a $y\in\Z$ such that
\begin{alignat}{2}
\begin{aligned}\label{hauptsatz 2 entartet - Kongruenzen y}
y&\equiv  b &&\mod a,\\
y&\equiv f_i&&\mod r_i,\text{ for }i=1,...,\ell
\end{aligned}
\end{alignat}
and all solutions are $n\equiv y\mod aR$.
By the first congruence in \eqref{hauptsatz 2 entartet - Kongruenzen x}
and \eqref{hauptsatz 2 entartet - Kongruenzen y},
there are $n_0,n_1\in\Z$ such that $x=cn_0+d$ and $y=an_1+b$.
Applying the Chinese remainder theorem again, we obtain
infinitely many $n\in\N$ such that
\begin{align*}
n\equiv n_0 &\mod Q,\\
n\equiv n_1 &\mod R.
\end{align*}
This yields us $u,v\in\Z$ such that $n=uQ+n_0=vR+n_1$. Therefore,
\begin{align*}
cn+d=ucQ+cn_0+d\equiv x&\mod cQ,\\
an+b=vaR+an_1+b\equiv y&\mod aR.
\end{align*}
Thus $cn+d$ solves the congruence system
\eqref{hauptsatz 2 entartet - Kongruenzen x} and $an+b$
solves \eqref{hauptsatz 2 entartet - Kongruenzen y}.
By \eqref{hauptsatz 2 entartet - Kongruenzen x},
we have $q_i\nmid cn+d$ if and only if $e_i=1$,
thus $q_i\mid t$ by \eqref{hauptsatz 2 entartet - Form von t,u}.
Therefore, we obtain $\tau_{cn+d}(\tau_c(a))=t$ and
thus $\tau_{cn+d}(a)=\tau_{cn+d}(\gamma_c(a))\cdot t$.
Using Lemma~\ref{Lemma fuer Zaehlen von phi Werten}i),
we obtain $\tau_{cn+d}(a)=\tau_d(\gamma_c(a))\cdot t$,
thus $n$ satisfies \eqref{hauptsatz 2 entartet - Form von tau a}.
Similarly, by \eqref{hauptsatz 2 entartet - Kongruenzen y},
we have $r_i\nmid an+b$ if and only if $f_i=1$,
thus $r_i\mid u$ by \eqref{hauptsatz 2 entartet - Form von t,u}.
Therefore, we obtain $\tau_{an+b}(\tau_a(c))=u$ and
thus $\tau_{an+b}(c)=\tau_{an+b}(\gamma_a(c))\cdot u=\tau_b(\gamma_a(c))\cdot u$,
which means that we have found infinitely many $n$
that satisfy \eqref{hauptsatz 2 entartet - Form von tau a}
and \eqref{hauptsatz 2 entartet - Form von tau c}.
Combining \eqref{hauptsatz 2 entartet - Form von tau a}
and \eqref{hauptsatz 2 entartet - Form von tau c}
with the fact that $(c,t)=(u,a)=1$, we can use the multiplicativity
of $\Phi$ to see that for every $n$,
\[
\frac{\Phi(\tau_{cn+d}(a))}
     {\Phi(\tau_{an+b}(c))}
=
\frac{\Phi(\tau_d(\gamma_c(a)))}
     {\Phi(\tau_b(\gamma_a(c)))}
\cdot
\frac{\Phi(t)}{\Phi(u)}
\]
for some $t\Mid\tau_c(a)$ and $u\Mid\tau_a(c)$.

We have thus proven that \eqref{eq:phibeh} holds.

Multiplying the elements on both sides of \eqref{eq:phibeh}
by $\frac{a}{c}$
and using \eqref{hauptsatz 2 entartet - phibruch zu Phibruch}, we obtain
\begin{equation}\label{hauptsatz 2 entartet - Geschlossene Form von L}
L=
\frac{a}{c}
\cdot
\frac{\Phi(\tau_d(\gamma_c(a)))}
     {\Phi(\tau_b(\gamma_a(c)))}
\cdot
\left\{
\frac{\Phi(t)}{\Phi(u)}
\,\middle|\,
t,u\in\N,\ t\Mid\tau_c(a),\ u\Mid\tau_a(c)
\right\},
\end{equation}
where $\frac{\phi(an+b)}{\phi(cn+d)}$ assumes every value on the right side for infinitely many $n$.
Using Lemma~\ref{Lemma fuer Zaehlen von phi Werten}(ii), we can see that two distinct pairs of $t,u$ in \eqref{hauptsatz 2 entartet - Geschlossene Form von L} yield a different value of $\frac{\Phi(t)}{\Phi(u)}$, hence
\[\#L=
\#\{t\in\N;\ t\Mid\tau_c(a)\}
\cdot
\#\{u\in\N;\ u\Mid\tau_a(c)\}.\]
By Lemma~\ref{Lemma fuer Zaehlen von phi Werten}(iii) and the additivity of \(\omega\), we conclude
\[
\#L
=2^{\omega(\tau_c(a))}\cdot 2^{\omega(\tau_a(c))}\\
=2^{\omega(\tau_c(a)\tau_a(c))}.
\]
\end{proof}

\begin{rem}\label{Rnote2}
Note that similarly to Corollary~\ref{Cor3.3}, we can easily obtain a generalization of Theorem~\ref{Thm4.2} where $n$ belongs to an arbitrary arithmetic progression $n\equiv r\mod m$. We do not give the details here.
\end{rem}

\section{Conditions for $M=\N_{>-b/a}$}
\label{sec5}

\subsection{A complete characterization}\label{sec5.1}
By Corollary~\ref{Thm3.4} we know that the density interval $]0,D]$
in Theorem~\ref{Thm3.1} is the
largest possible, since the sequence of fractions $\phi(an+b)/(cn+d)$
has no limit point in $]-\infty,-\epsilon[$ or $]D+\epsilon,\infty[$
for arbitrary $\epsilon>0$.
Hence the sequence %
cannot be dense in $]-2\epsilon,D]$ or $]0,D+2\epsilon]$.

In this section we observe that there are cases
in which the set $M$ in Theorem~\ref{Thm3.1} contains \emph{all}
positive integers.
The easiest case is $a=c=1$, $b=d=0$. Then $D=1$ and
\[0<\frac{\phi(n)}{n}\leq 1.\]
Hence the arithmetic function $\phi(n)/n$
does not hit any value outside of $]0,1]$, which implies that its image
over all $n\in\N$ is dense in $]0,1]$, i.\,e.\ $M=\N$.

In Theorem~\ref{Thm3.6} below we provide a complete answer
when this situation happens. For this, we introduce
the following notation.

\begin{deff}\label{Def3.5}
  Let $a\in\N$, $b\in\Z$ and put $g:=(a,b)$, $a':=a/g$, $b':=b/g$.
  Then define
\[R(a,b):=\{n\in\N_{>-b/a}\ ;\  \rad(a'n+b')\mid g\},\]
\[
R_0(a,b):=\{
n\in\Z_{>-b/a}\ ;\  
\forall p^k\| (a'n+b'):
p\mid g
,\:
k<\ord_{a'}(p)
\}
\]
and
\[\rho(a,b):=\#R(a,b),\ \rho_0(a,b):=\#R_0(a,b),\]
which may be $\infty$.
\end{deff}

Since $\rad(a'n+b')$ is not defined for $a'n+b'=0$, we assume $n>-b/a$,
i.e.\ $a'n+b'>0$. Note that $R_0(a,b)\cap\N\subseteq R(a,b)$.
Since $(a',a'n+b')=(a',b')=1$ by Lemma~\ref{Lem2.5}, we obtain that
all $p\mid(a'n+b')$ are coprime to $a'$, and thus $\ord_{a'}(p)$ occurring in
the definition of $R_0(a,b)$ exists.

We have $\rho_0(a,b)<\infty$ for all pairs $a,b$, since $g$ determines
an upper bound to the number of distinct prime divisors of $a'n+b'$, and
$\ord_{a'}(p)$ determines an upper bound to their multiplicity. Furthermore,
note that $\rad(a'n+b')\mid g$ is equivalent to $\tau_g(a'n+b')=1$.

\medskip
We provide a complete characterization of the condition
$M=\N_{>-b/a}$ in terms of $\rho(a,b)$ with the next theorem.
\begin{theorem}\label{Thm3.6}
  Let $a\in\N,b\in\Z,c\in\R_{\neq0},d\in\R$ be constants.
  Put $g:=(a,b)$ and $a':=a/g$, $b':=b/g$, $r_n:=a'n+b'$. Let
  \begin{multline*}
    M:=\left\{
n\in\N;\ 
n>\max\Big\{-\frac{b}{a},-\frac{d}{c}\Big\}
\text{ and }
\frac{\phi(an+b)}{cn+d}\leq\frac{a\phi(g)}{cg}
\right\},\\
R'(a,b):=\{n\in\N_{>-b/a};\ \tau_g(r_n)\geq2\},\\
P(a,b):=\{n\in\N;\ r_n\in\P\text{ and }r_n\nmid g\}.
\end{multline*}
Clearly $R(a,b)\cup R'(a,b)=\N_{>-b/a}$. Then
\begin{enumerate}[i)]
\item $\#M=\infty$ and
\[M=\left\{
n\in\N;\ 
n>-\frac{b}{a}
\text{ and }
n\geq\frac{\gamma_g(r_n)\phi(\tau_g(r_n))}{a'}-\frac{d}{c}
\right\}.\]
\item
$\#P(a,b)=\infty$ and $P(a,b)\subseteq R'(a,b)$. Furthermore,
\[M\cap R(a,b)\neq\emptyset\implies d/c\geq b/a\implies M=\N_{>-b/a}\]
and
\[M\cap P(a,b)\neq\emptyset\implies d/c\geq(b-g)/a\implies R'(a,b)\subseteq M.\]
\item
  $M=\N_{>-b/a}$ if and only if
\[
d/c\geq\begin{cases}
b/a,&\text{if }\rho(a,b)\geq1,\\
(b-g)/a,&\text{otherwise}.
\end{cases}
\]
\end{enumerate}
\end{theorem}

\begin{proof}
Since $d/c\in\R$ and
\[
\frac{\phi(an+b)}{cn+d}\leq\frac{a\phi(g)}{cg}
\iff\frac{\phi(an+b)}{n+d/c}\leq\frac{a\phi(g)}{g},
\]
we may replace $d/c$ by $d$ and assume $c=1$ without loss of generality.

\medskip
\textsc{Assertion }i):
$\#M=\infty$ follows immediately from Theorem~\ref{Thm3.1}.
Note that by Lemma~\ref{Thm2.3},
\begin{multline*}
\frac{n+d}{\phi(an+b)}\cdot\frac{a\phi(g)}{g}
=\frac{n+d}{\phi(r_ng)}\cdot a'\phi(g) \\
=\frac{n+d}{\gamma_g(r_n)\phi(\tau_g(r_n))\phi(g)}\cdot a'\phi(g)
=\frac{n+d}{\gamma_g(r_n)\phi(\tau_g(r_n))/a'}.
\end{multline*}
If $n+d>0$, we deduce
\begin{multline}\label{eqt1}
  \frac{\phi(an+b)}{n+d} \leq\frac{a\phi(g)}{g}\\
\iff 1\leq\frac{n+d}{\phi(an+b)}\cdot\frac{a\phi(g)}{g}
=\frac{n+d}{\gamma_g(r_n)\phi(\tau_g(r_n))/a'}\\
\iff \frac{\gamma_g(r_n)\phi(\tau_g(r_n))}{a'}\leq n+d. 
\end{multline}
Now define
\[M':=\left\{
n\in\N;\ n>-\frac{b}{a}\text{ and }
n\geq\frac{\gamma_g(r_n)\phi(\tau_g(r_n))}{a'}-d
\right\}.\]
Let $n\in M$. Then $n>\max\{-b/a,-d\}$, which implies $n+d>0$. Furthermore
\[\frac{\phi(an+b)}{n+d}\leq\frac{a\phi(g)}{g},\]
and \eqref{eqt1} yields $n\in M'$.
Now let $n\in M'$. Then $n>-b/a$ and
\[n+d\geq\frac{\gamma_g(r_n)\phi(\tau_g(r_n))}{a'}>0,\]
hence $n>-d$, and from \eqref{eqt1} we obtain $n\in M$.
Thus $M=M'$.

\medskip
\textsc{Assertion }ii):
By Dirichlet's prime number theorem there are infinitely many $n$ such that
$r_n=a'n+b'\in\P$ (if $b=0$ then $a'=1$). Since $g$ has only finitely
many prime factors, this implies that $\#P(a,b)=\infty$.

Let $n\in P(a,b)$. Then $(r_n,g)=1$ and thus $\tau_g(r_n)=r_n\geq2$.
Since $r_n=a'n+b'>0$, we get $n>-b'/a'=-b/a$, thus $n\in R'(a,b)$,
so that $P(a,b)\subseteq R'(a,b)$.

Assume that there is an $n\in M\cap R(a,b)$. Then $\gamma_g(r_n)=r_n$
and $\tau_g(r_n)=1$, and by i) we obtain
\[
d\geq\frac{\gamma_g(r_n)\phi(\tau_g(r_n))}{a'}-n
=\frac{r_n}{a'}-n
=\frac{b'}{a'}
=\frac{b}{a}.
\]
Now let $d\geq b/a$. Then for all $n\in\N_{>-b/a}$ we deduce
\[
\frac{\gamma_g(r_n)\phi(\tau_g(r_n))}{a'}-d
\leq\frac{\gamma_g(r_n)\tau_g(r_n)}{a'}-\frac{b}{a}
=\frac{r_n}{a'}-\frac{b'}{a'}
=n.
\]
Therefore $n\in M$ and thus $M=\N_{>-b/a}$.

Assume that there is an $n\in M\cap P(a,b)$. Then $\gamma_g(r_n)=1$
and $\tau_g(r_n)=r_n$, thus
\[
d\geq\frac{\gamma_g(r_n)\phi(\tau_g(r_n))}{a'}-n
=\frac{\phi(r_n)}{a'}-n
=\frac{r_n-1}{a'}-n
=\frac{b'-1}{a'}
=\frac{b-g}{a}.
\]
Now let $d\geq (b-g)/a$ and $n\in R'(a,b)$. Since $\tau_g(r_n)\geq2$
and $\gamma_g(r_n)\geq1$, we obtain
\begin{multline*}
\frac{\gamma_g(r_n)\phi(\tau_g(r_n))}{a'}-d
\leq\frac{\gamma_g(r_n)(\tau_g(r_n)-1)}{a'}-\frac{b-g}{a}\\
=\frac{r_n-\gamma_g(r_n)}{a'}-\frac{b'-1}{a'}
=n+\frac{1-\gamma_g(r_n)}{a'}
\leq n,
\end{multline*}
which shows that $n\in M$ and thus $R'(a,b)\subseteq M$.

\medskip
\textsc{Assertion }iii): Assume that $M=\N_{>-b/a}$. Since
$P(a,b)\neq\emptyset$, we get $M\cap P(a,b)\neq\emptyset$, and then ii)
implies $d/c\geq(b-g)/a$. If $R(a,b)\neq\emptyset$, then
$M\cap R(a,b)\neq\emptyset$ and $d/c\geq b/a$.

If $d/c\geq b/a$, then ii) implies $M=\N_{>-b/a}$.
If $d/c\geq(b-g)/a$ and $R(a,b)=\emptyset$, then $R'(a,b)=\N_{>-b/a}$
and thus $M=\N_{>-b/a}$.
\end{proof}

\subsection{Determination of $\rho(a,b)$} \label{sec5.2}
For a fixed given pair $a,b$ of integers, it is quite difficult to determine
an $n\in\N$ such that $\rad(a'n+b')\mid g$, i.e.\ $n\in R(a,b)$,
or show its nonexistence. The following supplementary theorem
allows us to determine the
value of $\rho(a,b)$, at least for some pairs $a,b$.
Theorem~\ref{Thm3.8} then allows us to calculate
$\rho(a,b)$ for any pair $a,b$, which we perform in Algorithm~\ref{Algo2}.

\begin{theorem}\label{Thm3.7}
  Let $a\in\N$, $b\in\Z$ and $g:=(a,b)$, $a':=a/g$, $b':=b/g$.
\begin{enumerate}[i)]%
\item %
  Assume $g\geq2$ and that at least one of the following
  cases is true.
\begin{enumerate}[(1)]
\item $a=g$,
\item $\rad(g)\nmid a'$, $b=g$,
\item $b>g$, $\rad(b)\mid g$.
\end{enumerate}
Then $\rho(a,b)=\infty$. Expressed in terms of $\rho$, this is equivalent to
\begin{enumerate}[(1)]
\item $\rho(a,ak)=\infty$ for $a\geq2$ and any $k\in\Z$,
\item $\rho(kb,b)=\infty$ for $b\geq2$ and any $k\in\N$
  such that $\rad(b)\nmid k$,
\item $\rho(a,b)=\infty$ for $b>g$ and $\rad(b)\mid g$.
\end{enumerate}
\item 
  There is an $n\in\N_{>-b/a}$ such that $a'n+b'=1$ if and only if
  $b'\leq0$ and $a'\mid1-b'$. If $\rad(g)\mid a'$, this is the only
  $n\in\N_{>-b/a}$ such that $\rad(a'n+b')\mid g$, hence
\[
\rho(a,b)=
\begin{cases}
1,&\text{if }b'\leq0\text{ and }a'\mid1-b',\\
0,&\text{otherwise}.
\end{cases}
\]
\end{enumerate}
\end{theorem}

\begin{proof}
\textsc{Assertion }i):
$\bullet$ Case (1):
Assume $a=g\geq2$, then $a'=1$.

Put $n=g^k-b'$ for any $k\in\N$ such that $g^k>b'$ and $n>-b/a=-b'/a'$.
Then $a'n+b'=g^k$ and thus $\rad(a'n+b')=\rad(g)\mid g$.

$\bullet$ Case (2):
Assume $\rad(g)\nmid a'$ and $b=g\geq2$.

There is a $p\mid g$ such that $p\nmid a'$, which implies
$(p,a')=1$. Define $k:=\ord_{a'}(p)\in\N$ to be the multiplicative order
of $p$ modulo $a'$. Then $p^{hk}\equiv1\mod a'$ for all $h\in\N$.
With this, put $n:=(p^{hk}-1)/a'\in\N$ for any $h\in\N$ with $n>-b/a$. 
Then
\[\rad(a'n+1)=\rad(p^{hk})=p\mid g.\]

$\bullet$ Case (3):
Assume $b>g$ and $\rad(b)\mid g$.

Since $(a',b')=1$, we define $k:=h\cdot\ord_{a'}(b')+1\geq2$
for any $h\in\N$. Then $(b')^k=b'\mod a'$ and thus
\[a'\mid (b')^k-b'.\]
Put $n:=((b')^k-b')/a'\in\N$ (because $b'>1$). 
Then $a'n+b'=(b')^k$
and thus $\rad(a'n+b')=\rad(b')\mid\rad(b)\mid g$.

\medskip
\textsc{Assertion }ii): If $n>-b/a=-b'/a'$, then $a'n+b'>0$,
and thus $a'n+b'\geq1$, since this number is an integer. There is
an $n\in\N$ with $a'n+b'=1$ if and only if $n=(1-b')/a'\in\N$, which
means $1-b'\geq1$ and $a'\mid 1-b'$.

By Lemma~\ref{Lem2.5}, we have $(a',a'n+b')=(a',b')=1$ for all
$n\in\N_{>-b/a}$. If $\rad(g)\mid a'$, then $(g,a'n+b')=1$, 
hence $\rad(a'n+b')\mid g$ can only be true if $a'n+b'=1$.\end{proof}

\begin{example}\label{Exa1}
  Theorem~\ref{Thm3.7} does not yet cover all possible 
  pairs $a,b$. For example, no
statement is given when $a=14$, $b=6$, since then
$g=2$, $a'=7$, $b'=3$.
\end{example}

However, by Theorem~\ref{Thm3.7}ii),
we understand exactly when $n\in\N_{>-b/a}$ with $a'n+b'=1$
exists. Since $1$ divides any $g$, it suffices
to examine the question when $n\in\N_{>-b/a}$ with
$\rad(a'n+b')\mid g$ and $a'n+b'\geq2$ exists.
The following Theorem~\ref{Thm3.8}
reduces this problem to a finite question for all $a\in\N$, $b\in\Z$.

\begin{theorem}\label{Thm3.8}
  Let $a\in\N$, $b\in\Z$ and $g:=(a,b)$, $a':=a/g$, $b':=b/g$.
  Assume that $\tau_{a'}(\rad(g))$ has the prime factorization
  $\tau_{a'}(\rad(g))=q_1\cdots q_s$. Then the following statements
  are equivalent.
\begin{enumerate}
\item $\rho(a,b)=\infty$,
\item There is an $n\in\N_{>-b/a}$ such that $\rad(a'n+b')\mid g$ and
  $a'n+b'\geq 2$,
\item There are $\alpha_1,\dots,\alpha_s\in\N_0$ such that
  $\alpha_1+\cdots+\alpha_s>0$ and
\[q_1^{\alpha_1}\cdots q_s^{\alpha_s}\equiv b'\mod a',\]
\item $\rad(g)\nmid a'$ and there are $\alpha_1,\dots,\alpha_s\in\N_0$ such that
\[q_1^{\alpha_1}\cdots q_s^{\alpha_s}\equiv b'\mod a'.\]
\end{enumerate}
\end{theorem}

\begin{proof}
  If $\rho(a,b)=\infty$, there exist infinitely many
  $n\in\N_{>-b/a}$ such that $\rad(a'n+b')\mid g$, hence there exists
  one with $a'n+b'\geq2$ since $a'\geq1$.

  Assume there exists an $n\in\N_{>-b/a}$ such that $\rad(a'n+b')\mid g$ and
  $a'n+b'\geq2$. Due to Lemma~\ref{Lem2.5}, $(a',a'n+b')=(a',b')=1$,
  thus we obtain $\rad(a'n+b')\mid\tau_{a'}(\rad(g))$. Hence
  $a'n+b'=q_1^{\alpha_1}\cdots q_s^{\alpha_s}\geq2$ for some %
  $\alpha_1,\dots,\alpha_s\in\N_0$. Since
  $q_1^{\alpha_1}\cdots q_s^{\alpha_s}\geq2$, we have
  $\alpha_1+\cdots+\alpha_s>0$ and
  $q_1^{\alpha_1}\cdots q_s^{\alpha_s}\equiv b'\mod a'$.

  Assume that there are $\alpha_1,\dots,\alpha_s\in\N_0$
  such that $\alpha_1+\cdots+\alpha_s>0$ and
  $q_1^{\alpha_1}\cdots q_s^{\alpha_s}\equiv b'\mod a'$. Then there exists
  an $m\in\Z$ such that
\[
r:=a'm+b'=q_1^{\alpha_1}\cdots q_s^{\alpha_s},
\]
hence $\rad(r)\mid\tau_{a'}(\rad(g))$. Since $\alpha_1+\cdots+\alpha_s>0$,
we deduce $r\geq2$ and $\tau_{a'}(\rad(g))\geq2$, which implies
$\rad(g)\nmid a'$.

Finally, assume $\rad(g)\nmid a'$ and that there exist
$\alpha_1,\dots,\alpha_s\in\N_0$ such that
$q_1^{\alpha_1}\cdots q_s^{\alpha_s}\equiv b'\mod a'$.
We show that there are infinitely many $n\in\N_{>-b/a}$ such that
$\rad(a'n+b')\mid g$, i.e.\ $\rho(a,b)=\infty$. Let $r\in\N$ and $m\in\Z$
be as above. First, we assume $r=1$, which implies $b'=1-a'm$.
By the assumption $\rad(g)\nmid a'$, there exists $p\mid g$ such that
$p\nmid a'$ and thus $(p,a')=1$. Put $k:=\ord_{a'}(p)\in\N$, then
$p^{hk}\equiv1\mod a'$ for all $h\in\N$, thus $a'\mid p^{hk}-1$.
Define $\ell_h:=(p^{hk}-1)/a'\in\N$, with $h$ being chosen sufficiently
large such that $\ell_h+m>0$ and $\ell_h+m>-b/a$. Then
\[
p^{hk}
=a'\ell_h+1
=a'\ell_h+a'm+1-a'm
=a'(\ell_h+m)+b'.
\]
For $n:=\ell_h+m\in\N_{>-b/a}$ we deduce $\rad(a'n+b')=p\mid g$, and
since $h$ can be arbitrarily large, we obtain infinitely many $n$ with
this property, hence $\rho(a,b)=\infty$.

Now let $r\geq2$. Since $(a',r)=1$, we define $k:=\ord_{a'}(r)\in\N$.
Then $r^{hk}\equiv1\mod a'$ for all $h\in\N$, hence $a'\mid r^{hk+1}-r$.
Using this, define $\ell_h:=(r^{hk+1}-r)/a'\in\N$, where $h$ is chosen
sufficiently large such that $\ell_h+m>0$ and $\ell_h+m>-b/a$, which
is possible because of $r\geq2$. Then
\[
r^{hk+1}
=a'\ell_h+r
=a'\ell_h+a'm+b'
=a'(\ell_h+m)+b'.
\]
Put $n:=\ell_h+m\in\N_{>-b/a}$, then $\rad(a'n+b')=\rad(r)\mid g$.
Since $h$ can be arbitrarily large, we obtain infinitely many $n$
with this property, hence $\rho(a,b)=\infty$.\end{proof}

Let $a\in\N,b\in\Z$ be integers such that $\rho(a,b)\geq2$.
Then there exist positive integers $n_1>n_2>-b/a$ such that
$\rad(a'n_i+b')\mid g$ for $i=1,2$. Now $a'n_1+b'>a'n_2+b'>0$, 
thus $a'n_1+b'\geq2$, and Theorem~\ref{Thm3.8}
yields $\rho(a,b)=\infty$. This implies $\rho(a,b)\in\{0,1,\infty\}$
for all $a\in\N,b\in\Z$. %
From Examples~\ref{Exa1b}, \ref{Exa2} and \ref{Exa3} below,
we know that $\rho$ actually hits these three values.

Note that $a'$ and $b'$ do not correlate with $g$. It is only required that
$a'$ and $b'$ are coprime, where $g$ can be chosen independently.
In retrospect, we define then $a:=a'g$ and $b:=b'g$.

For every $1\leq i\leq s$, the sequence $q_i^0,q_i^1,q_i^2,q_i^3,\dots$
repeats modulo $a'$ with period $\ord_{a'}(q_i)$.
The question
if there exists a tuple $(\alpha_1,\dots,\alpha_s)\in\N_0^{s}$ such that
$q_1^{\alpha_1}\cdots q_s^{\alpha_s}\equiv b'\mod a'$ can therefore
be answered in finitely many steps using the following algorithm.

\begin{algo}\label{Algo1}
\medskip
\texttt{Input:} $a\in\N,b\in\Z$

\texttt{Define} $g:=(a,b)$, $a':=a/g$, $b':=b/g$.

\texttt{Determine prime factorization} $\tau_{a'}(\rad(g))=q_1\cdots q_s$.

\texttt{For} $0\leq\alpha_1<\ord_{a'}(q_1)$, \dots, 
$0\leq\alpha_s<\ord_{a'}(q_s)$:

\quad \texttt{If} $q_1^{\alpha_1}\cdots q_s^{\alpha_s}\equiv b'\mod a'$:

\quad\quad \texttt{Return true.}

\texttt{Return false.}
\end{algo}

\medskip
Note that in the case $\rad(g)\mid a'$ we have $s=0$
and Algorithm~\ref{Algo1} jumps immediately to the last line
and returns false.

\medskip
We intend to run this algorithm for many pairs $a,b$, and we 
improve upon the running time significantly if we store every possible value
of $q_1^{\alpha_1}\cdots q_s^{\alpha_s}\mod a'$ for each pair
$a',\tau_{a'}(\rad(g))$. If it happens that we stumble upon
the same pair $a',\tau_{a'}(\rad(g))$ again, we simply check whether
$b'\mod a'$ is contained in the previously stored set.

\bigskip%
The next algorithm determines $\rho(a,b)$ for all $a\in\N$, $b\in\Z$,
due to Theorem~\ref{Thm3.7}ii).

\begin{algo}\label{Algo2}

\texttt{Input:} $a\in\N$, $b\in\Z$

\texttt{Define} $g:=(a,b)$, $a':=a/g$, $b':=b/g$.

\texttt{Determine prime factorization} $\tau_{a'}(\rad(g))=q_1\cdots q_s$.

\texttt{If} $\rad(g)\nmid a'$ \texttt{and there are}
$\alpha_1,\dots,\alpha_s\in\N_0$ \texttt{with}
$q_1^{\alpha_1}\cdots q_s^{\alpha_s}\equiv b'\mod a'$:

\quad\texttt{Return} $\rho(a,b)=\infty$.

\texttt{If} $b'\leq0$ and $a'\mid1-b'$:

\quad\texttt{Return} $\rho(a,b)=1$.

\texttt{Return} $\rho(a,b)=0$.
\end{algo}

\begin{example}\label{Exa1b}
Recalling our previous example $a=14$, $b=6$ in Example~\ref{Exa1},
we detect an answer to the question raised there.
We have $g=2$, $a'=7$, $b'=3$, but
\[
2^1\equiv2\mod7,\qquad
2^2\equiv4\mod7,\qquad
2^3\equiv1\mod7.
\]
Thus $\rho(14,6)=0$, so there exists no $n\in\N$
with $\rad(7n+3)\mid2$.
\end{example}

\begin{example}\label{Exa2}
  Consider the example $a=1$, $b=0$. Then $g=1$, $a'=1$, $b'=0$
and $\rad(g)\mid a'$, thus $\rho(1,0)=1$ by Theorem~\ref{Thm3.7}ii).
\end{example}
\begin{example}\label{Exa3}
Finally, let $a=6$, $b=10$. Then $g=2$, $a'=3$, $b'=5$ and thus
$g\equiv b'\mod a'$. This means $\rho(6,10)=\infty$, hence there
are infinitely many $n\in\N$ such that $\rad(3n+5)\mid2$.
The sequence of these $n$ starts with $1,9,41,169,681,\dots$, and it
can be shown that its $k$-th element is equal to $(2^{2k+1}-5)/3$.
Its growth can be described more precisely
by Theorem~\ref{Thm3.11}.
\end{example}

Figure \ref{Fig1} illustrates $\rho(a,b)$
for all $1\leq a\leq2001$ and $-1000\leq b\leq1000$ with a black
pixel when $\rho(a,b)=\infty$, a red pixel when $\rho(a,b)=1$,
and a white pixel when $\rho(a,b)=0$. 
%
%
\begin{figure}\caption{Points with $\rho(a,b)=\infty$ (black),
    $\rho(a,b)=1$ (red) and $\rho(a,b)=0$ (white)}
\label{Fig1}
\begin{center}
\pgfkeys{/pgf/number format/1000 sep={}}
\begin{tikzpicture}
\begin{axis}
[xlabel=$a$,ylabel=$b$,
width=350,height=350,
enlargelimits=false,
axis lines=left,
ylabel style={rotate=-90}]
\addplot graphics
[xmin=1,xmax=2001,ymin=-1000,ymax=1000]
{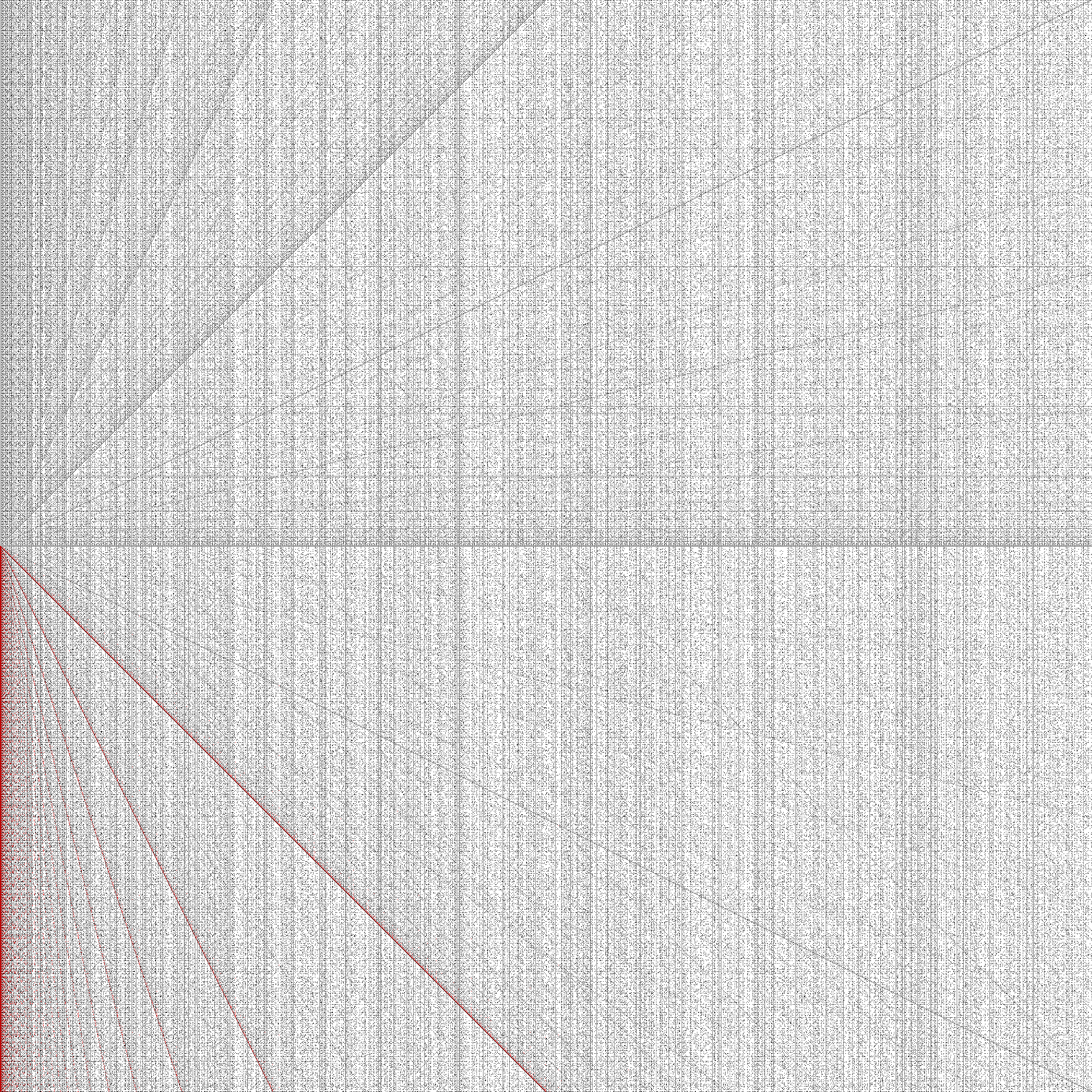};
\end{axis}
\end{tikzpicture}
\end{center}
\end{figure}

The most striking feature of this diagram are the dotted lines
where $a\geq2$ and $b=ka$ for some $k\in\Z$, which illustrates
Theorem~\ref{Thm3.7}i). 

Furthermore, we clearly observe three dotted lines emerging at 45 degrees
from each other, originating at the $b$-axis where $b\in\{128,256,512\}$.
These visible lines emerge for any prime power $p^k$
and they become thinner for larger primes $p$.
An explanation to this phenomenon is due to the
following Theorem.

\begin{theorem}\label{Thm3.9}
  For all primes $p$ and $a,k\in\N$, $x\in\Z$,
\[
\rho(a,p^k+ax)=
\begin{cases}
\infty,&\text{if }p\mid a\text{ and }p^{k+1}\nmid a,\\
1,&\text{if }x\leq -p^k/a\text{ and }a\mid p^k-1,\\
1,&\text{if }x\leq -p^k/a\text{ and }p^{k+1}\mid a,\\
0,&\text{else}.
\end{cases}
\]
\end{theorem}

\begin{proof}
  Let $b=p^k+ax$ and let $\ell\in\N_0$ be the largest integer
  such that $p^\ell\mid a$. This yields $b-ax=p^k$ and thus
  $g\mid p^k$. Note that if $\ell\leq k$, then $p^\ell\mid b$
  and thus $g=p^\ell$.

  If $1\leq\ell\leq k$, then $g=p^\ell\geq2$ and
  $b'=p^{k-\ell}+a'x\equiv p^{k-\ell}\mod a'$. Now since
  $\rad(g)=p\nmid a'$ and $k-\ell\geq0$, we obtain by
  Theorem~\ref{Thm3.8} that $\rho(a,b)=\infty$.

  If $\ell=0$, then $p\nmid a$, so $g=1$, $a=a'$, $b=b'$ and by
  Theorem~\ref{Thm3.7}ii), we have $\rho(a,b)=1$ if $b\leq0$
  and $a\mid 1-b=1-p^k-ax$, which is equivalent to $p^k\leq-ax$
  and $a\mid p^k-1$, otherwise we have $\rho(a,b)=0$.
  Note that from $a\mid p^k-1$, it already follows that $p\nmid a$.

  If $\ell>k$, then $g=p^k$ and $\rad(g)=p\mid a'$ and
  by Theorem~\ref{Thm3.7}ii) it follows that $\rho(a,b)=1$
  if $b'=1+a'x\leq0$ and $a'\mid 1-b'=-a'x$, which is equivalent
  to $x\leq -1/a'$. Otherwise we have $\rho(a,b)=0$.
  This completes the proof.\end{proof}

The case $p=2$ and $x\in\{-1,0,1\}$ then gives us the 
visible dotted lines in the graph emerging at the $b$-axis
where $b$ is a power of 2.

\subsection{Growth of the counting function of $R(a,b)$}\label{sec5.3}
We have now seen that $\rho(a,b)=\infty$ for infinitely
many pairs $a,b$. Our next question is then: How frequent are
the $n\in R(a,b)$, that is, how quickly does
the counting function $\#\{n\in R(a,b);\ n\leq x\}$
grow when $x\to\infty$?

\bigskip
We obtain the following asymptotic formula %
in terms of the function $\rho_0(a,b)\in\N_0$ from Definition~\ref{Def3.5}.

 \begin{theorem}\label{Thm3.11}
   Let $a\in\N$, $b\in\Z$ and put $g:=(a,b)$, $a':=a/g$, $b':=b/g$.
   Let %
   $\tau_{a'}(\rad(g))=q_1\cdots q_s$ be the prime factorization of
   $\tau_{a'}(\rad(g))$. For $1\leq i\leq s$, let
   $k_i:=\ord_{a'}(q_i)$. If $\rho(a,b)=\infty$, then $s\geq1$,
   $\rho_0(a,b)\geq1$, and for $x\to\infty$ we have 
 \[
 \#\{n\in R(a,b);\ n\leq x\}
 =\frac{\rho_0(a,b)}{s!\log(q_1^{k_1})\cdots\log(q_s^{k_s})}\cdot\log(x)^s
 +O(\log(x)^{s-1}).
 \]
 \end{theorem}

\begin{proof}
  Since $\rho(a,b)=\infty$, Theorem~\ref{Thm3.8} yields $\rad(g)\nmid a'$
  and thus $\tau_{a'}(\rad(g))\geq2$ and $s\geq1$. Let $n\in R(a,b)$.
  By Lemma~\ref{Lem2.5}, $(a',a'n+b')=(a',b')=1$, and hence
  $\rad(a'n+b')\mid\tau_{a'}(\rad(g))$. This means that there are
  integers $\alpha_1,\dots,\alpha_s\in\N_0$ such that
\[a'n+b'=\prod_{i=1}^s q_i^{\alpha_i}.\]
Since $q_i^{k_i}\equiv1\mod a'$, we obtain
\[b'\equiv\prod_{i=1}^s q_i^{\alpha_i+k_ih_i} \mod a'\]
for all $h_i\in\Z$ such that $\alpha_{i}+k_{i}h_{i}\geq 0$.

Therefore, there is a unique $m\in\Z$ such that
\[a'm+b'=\prod_{i=1}^s q_i^{\beta_i}\]
with $\beta_i\equiv\alpha_i\mod k_i$ and $0\leq\beta_i<k_i$
for all $1\leq i\leq s$. Note that $m>-b'/a'$, which implies
$m\in R_0(a,b)$. This induces a map $f:R(a,b)\to R_0(a,b)$ that we
use to split $R(a,b)$ into pairwise disjoint sets, i.e.\ the
inverse image $f^{-1}(m)$ of each $m\in R_0(a,b)$, namely
\[R(a,b)=\dot{\bigcup_{m\in R_0(a,b)}}f^{-1}(m).\]
Since these sets are disjoint, we obtain
\begin{equation}\label{e1}
  \#\{n\in R(a,b);\ n\leq x\}=\sum_{m\in R_0(a,b)}
    \#\{n\in f^{-1}(m);\ n\leq x\}.
\end{equation}
Let $n\in f^{-1}(m)$. Then $n>0$ and
$a'n+b'=(a'm+b')\prod_{i=1}^s q_i^{k_ih_i}$ for some
$h_1,\dots,h_s\in\N_0$, thus
\begin{multline*}
f^{-1}(m)=\Big\{
\frac{(a'm+b')\prod_{i=1}^s q_i^{k_ih_i}-b'}{a'};\ 
h_1,\dots,h_s\in\N_0,\\
\prod_{i=1}^s q_i^{k_ih_i}>\frac{b'}{a'm+b'}
\Big\}
\end{multline*}
for all $m\in R_0(a,b)$. The condition
$\prod_{i=1}^s q_i^{k_ih_i}>\frac{b'}{a'm+b'}$ ensures
that $f^{-1}(m)$ contains only positive integers,
but clearly, this condition only excludes finitely many values.
This means that for $x\to\infty$,
\begin{equation}\label{e2}
  \#\{n\in f^{-1}(m);\ n\leq x\}=F(m,x)+O(1),
\end{equation}
where
\begin{multline*}
F(m,x):=
\#\left\{
\frac{(a'm+b')\prod_{i=1}^s q_i^{k_ih_i}-b'}{a'}\leq x;\ 
h_1,\dots,h_s\in\N_0
\right\}\\
=
\#\left\{
\prod_{i=1}^s q_i^{k_ih_i}\leq \frac{a'x+b'}{a'm+b'};\
h_1,\dots,h_s\in\N_0
\right\}.
\end{multline*}

In this situation, we apply Lemma~\ref{LemTen}
with $k=s$, $z=\log\frac{a'x+b'}{a'm+b'}$,
$a_{i}=k_{i}\log q_{i}=\log(q_{i}^{k_{i}})$, therefore we obtain
\begin{equation}\label{e3}
F(m,x)
=
\frac{\log\left(\frac{a'x+b'}{a'm+b'}\right)^s}
{s!\log(q_1^{k_1})\cdots\log(q_s^{k_s})}
+O\left(\log\left(\frac{a'x+b'}{a'm+b'}\right)^{s-1}\right),
\end{equation}
note that the implicit constant depends on the variables
$q_{1},\dots,q_{s}$, $k_{1},\dots,k_{s}$ which in turn only depend
on $a$ and $b$ and are thus independent of $x$.
Simplifying the argument of the logarithm yields
\begin{equation*}
\log(a'x+b')^s
=\log(x)^s + O(\log(x)^{s-1})
\end{equation*}
and thus
\begin{equation*}
\log\left(\frac{a'x+b'}{a'm+b'}\right)^s
=\log(x)^s + O(\log(x)^{s-1}).
\end{equation*}
From \eqref{e3}, we obtain
\[
F(m,x)
=
\frac{\log(x)^s}
{s!\log(q_1^{k_1})\cdots\log(q_s^{k_s})}
+O(\log(x)^{s-1}).
\]
Then \eqref{e2} yields
\[
\#\{n\in f^{-1}(m);\ n\leq x\}=
\frac{\log(x)^s}
{s!\log(q_1^{k_1})\cdots\log(q_s^{k_s})}
+O(\log(x)^{s-1}).
\]
Note that the terms in this asymptotic formula do not depend on $m$.
Finally, by \eqref{e1} and $\#R_0(a,b)=\rho_0(a,b)$ we deduce
\[
\#\{n\in R(a,b);\ n\leq x\}
=\frac{\rho_0(a,b)}
{s!\log(q_1^{k_1})\cdots\log(q_s^{k_s})}
\cdot\log(x)^s
+O(\log(x)^{s-1}),
\]
which was to be proved.\end{proof}

\subsection{The number of pairs $a,b$ with $\rho(a,b)=\infty$}\label{sec5.4}
\bigskip
Finally, we discuss in this last section the number of
pairs $1\leq a,b\leq x$ with $\rho(a,b)=\infty$, which leads us to
an interesting open question. Thus we examine the sum
\begin{equation}\label{eq:upb}
  \sum_{\substack{1\leq a,b\leq x\\\rho(a,b)=\infty}} 1 \leq \sum_{1\leq a,b\leq x} 1 \leq x^{2}.
\end{equation}

Let $g>1$ be a fixed integer.
The primes $q_i$ with $q_1\cdots q_s=\tau_{a'}(\rad(g))$ can be considered as
being independent of $a'$, since if we fix the prime divisors
$q_{1},\dots,q_{s}$ of $g$, then the positive integers $a'$ need to be taken with
$(a',q_1\cdots q_s)=1$. 
Hence, we obtain the lower bound
\begin{equation}\label{eq:verallgArn}
\sum_{\substack{1\leq a,b\leq x\\ \rho(a,b)=\infty}} 1
\geq \sum_{\substack{a'\leq x/g\\ (a',g)=1}}
\#\{q_1^{\alpha_1}\cdots q_s^{\alpha_s} \mod a';\
\alpha_1,\dots,\alpha_s\in\N_0\}.
\end{equation}

Since the terms %
$q_{1}^{\alpha_{1}}\dots q_{s}^{\alpha_{s}}$ generate a subgroup
of the multiplicative group mod $a'$
that contains all powers $q_1^{\alpha_1}$, 
we continue the lower estimate with
\begin{align}
  \sum_{\substack{1\leq a,b\leq x\\ \rho(a,b)=\infty}} 1
  &\geq \sum_{\substack{a'\leq x/g\\ (a',q_{1}\cdots q_{s})=1}}
  \#\{q_1^{\alpha_1}\mod a';\ \alpha_1\in\N_0\} \label{eq:lowest} \\
  &= \sum_{\substack{a'\leq x/g\\(a',q_{1})=1}}
  \#\{q_1^{\alpha_1}\mod a';\ \alpha_1\in\N_0\}
  \sum_{d\mid(a',q_{2}\cdots q_{s})} \mu(d) \notag \\
  &=\sum_{d\mid q_{2}\cdots q_{s}} \mu(d)
  \sum_{\substack{\tilde{a}\leq x/dg\\(\tilde{a},q_{1})=1}}
  \#\{q_1^{\alpha_1}\mod (d\tilde{a});\ \alpha_1\in\N_0\} \notag \\
    &= \sum_{d\mid q_{2}\cdots q_{s}} \mu(d) 
 \sum_{\substack{\tilde{a}\leq x/dg\\(\tilde{a},q_{1})=1}}
  \ord_{d\tilde{a}}(q_{1}).  \label{eq:ord1}
\end{align}

Note that in the special case when $s=1$,
this calculation simplifies to
\begin{equation}\label{eq:ord2}
  \sum_{\substack{1\leq a,b\leq x\\ \rho(a,b)=\infty}} 1 \geq
  \sum_{\substack{\tilde{a}\leq x/g\\(\tilde{a},q_{1})=1}} \ord_{\tilde{a}}(q_{1}).
\end{equation}
But when $s>1$, the necessary
condition $(a',q_1\cdots q_s)=1$ on the right hand side of
\eqref{eq:lowest} is more restrictive than $(a',q_1)=1$,
thus we may not simplify like this in general.

\medskip
The last sum appearing in \eqref{eq:ord1} and \eqref{eq:ord2}
has been studied in the literature, but considered 
with mean value in $q_1$. The growth of the average
order of elements (and similarly, of one element taken for many moduli)
is a famous question raised by Arnold in \cite{Arn},
see also \cite{KurPom}.

From the results of Luca and Shparlinski
in \cite{LucShp}, namely the asymptotic formula
\[
  \frac{1}{x}\sum_{n<x}\frac{1}{\phi(n)}\sum_{\substack{q_{1}<n\\(n,q_{1})=1}}
  \ord_{n}(q_{1})=\frac{x}{\log x} \log((1+o(1))b\log\log x/\log\log\log x),
\]
we deduce with $\phi(n)\gg n/\log\log n$ that 
\begin{multline*}
  \frac{1}{x}\sum_{q_{1}\leq x}\sum_{\substack{q_{1}<n\leq x\\(n,q_{1})=1}} \ord_{n}(q_{1})
  \geq \sum_{x/2<n<x} \phi(n) \cdot\frac{1}{x}\cdot \frac{1}{\phi(n)}
  \sum_{\substack{q_{1}<n\\(n,q_{1})=1}} \ord_{n}(q_{1}) \\
  \gg \frac{x}{\log\log x}\cdot \frac{x}{\log(x)}
  \log(b\log\log x/\log\log\log x) \\
  = \frac{x^2}{\log x \log\log x} \log(b\log\log x/\log\log\log x)
\end{multline*}
for some constant $b>0$.
Thus, in the mean value over $q_1$, we get a lower bound pretty close
to $x^2$. A successor paper of Kim \cite{Kim} shows that this lower bound
is also admissible when taking a mean value over $q_1$ in the shorter range
$q_{1}\leq y$ with $x^{1-\delta}=o(y)$ for some $\delta>0$. But it seems that
no good result is known if $q_{1}$ is fixed.

Consider the case that $a'$ is a prime such that $q_{1}$ is a primitive
root mod $a'$. Then the powers $q_{1}^{\alpha_{1}}$ run through all reduced
residues mod $a'$, so that $\ord_{a'}(q_{1})=a'-1$. This yields
for large $x$ that
\begin{multline*}
  \sum_{\substack{1\leq a,b\leq x\\\rho(a,b)=\infty}} 1
  \geq \sum_{1\leq a\leq x} \sum_{\substack{1\leq b\leq a\\ \rho(a,b)=\infty}}1
  \geq
  \sum_{\substack{x/2g\leq a'\leq x/g\\ a'\text{prime} \\q_{1} \text{ primitive root} \mod a'}}
  (a'-1) \\ \gg x\sum_{\substack{x/2g<a'\leq x/g\\q_1 \text{ primitive root} \mod a'}}1
  \gg  \frac{x^{2}}{\log x},
\end{multline*}
assuming in the last step Artin's primitive root conjecture which states that
the number of primes $\leq x$ for which some $q$ 
is a primitive root, is $=(1+o(1))x/\log x$.

\begin{openq}\label{openq}
Thus we know that the sum
\[
  \sum_{\substack{a\leq x\\ (a,q_{1}\cdots q_{s})=1}}
  \#\{q_1^{\alpha_1}\cdots q_s^{\alpha_s} \mod a;
 \ \alpha_1,\dots,\alpha_s\in\N_0\}
\]
is trivially bounded by $\ll x^2$, and heuristically by
$\gg x^2/\log x$ from unsolved conjectures. The summand is the order of the
subgroup generated by $q_{1},\dots,q_{s}$ in the multiplicative group mod $a$.
It remains to be an \emph{open question} if there exists an asymptotic
  formula for this quantity, or even to decide if it is $=o(x^2)$. 
  This question can be seen as a generalization of Arnold's problem
  in \cite{Arn}, cp.~\cite{KurPom}.
\end{openq}
 
\section*{Acknowledgements}
We thank the referee for a careful reading and many
valuable suggestions on the first version of the manuscript. 


\end{document}